\documentclass[11pt]{amsart}
\usepackage{amsmath}
\usepackage[dvips]{graphicx}
\usepackage{amsfonts}
\usepackage{amsopn}
\usepackage{amsthm} 
\usepackage{amssymb}
\usepackage{color}
\usepackage{hyperref}

\usepackage{latexsym}
\usepackage{amsgen}

\newtheorem{theorem}{Theorem}[section]
\newtheorem{definition}[theorem]{Definition}
\newtheorem{lemma}[theorem]{Lemma}
\newtheorem{corollary}[theorem]{Corollary}
\newtheorem{proposition}[theorem]{Proposition}

\newtheorem{example}[theorem]{Example}
\newtheorem{remark}[theorem]{Remark}

\numberwithin{equation}{section}

\def\divrho{{\rm div}_\rho}

\def\kerGp{{\rm ker\,}(\nabla_G)^\perp}

\def\ow{{ \sqrt{\omega_{ij}}}}
\def\oww{{ {\omega_{ij}}}}

\usepackage{amssymb} 
\title{Geodesic of minimal length in the set of probability measures on graphs}
\author[Gangbo]{Wilfrid Gangbo}
\email{wgangbo(wcli,muchenchen)@math.ucla.edu}
\author[Li]{Wuchen Li}
\author[Mou]{Chenchen Mou}
\address{Department of Mathematics, UCLA, Los Angeles, CA 90095}

\begin{document}
\maketitle
\begin{abstract} We endow the set of probability measures on a weighted graph with a Monge--Kantorovich metric, induced by a function defined on the set of vertices. The graph is  assumed to have $n$ vertices and so, the boundary of the probability simplex is an affine $(n-2)$--chain.  Characterizing the geodesics of minimal length which may intersect the boundary, is a challenge we overcome even when  the endpoints of the geodesics  don't share the same connected components. It is our hope that this work would be a preamble to the theory of Mean Field Games on graphs.
\end{abstract}
%
%
\section{Introduction} 
The past two decades have witnessed an increasing number of studies on geodesics of minimal length on the set of probability measures on manifolds and Hilbert spaces \cite{AmbrosioGS} \cite{vil2003} (cf. e.g. for applications \cite{Gangbo3} \cite{Cardaliaguet} \cite{CardaliaguetDLL} \cite{Gangbo0} \cite{Gangbo1} \cite{Gangbo2} \cite{GangboT10}).  In these cases, the geodesics are characterized  by Hamilton--Jacobi equations which appear through a duality argument (cf. e.g. \cite{AmbrosioGS} \cite{EvansG} \cite{Gangbo95} \cite{GangboM95} \cite{GangboM96} \cite{vil2003}). The story is different when Hilbert spaces or manifolds are replaced by discrete states which are not length spaces. For practical reasons (e.g. computational reasons \cite{chow2012} \cite{li_2017} \cite{li_SE}),  one faces the issue of dealing with  geodesics of minimal length on the set of probability measures on graphs, the probability simplexes. Therefore, one needs to go beyond understanding the differential structure of the interior of probability simplexes, which is a rather simple task, and push the study to the boundary. Indeed, the discrete counterpart of prior studies on length spaces such as $\mathbb R^d$, turns out to be awfully complicated on probability simplexes, when the geodesics  contain boundary points. The goal of this manuscript is the study of geodesics of various metrics on the probability simplexes, without excluding the possibility that the complement of the endpoints  touch the boundary.

Let $G=(V, E, \omega)$ denote an undirected graph  of vertices $V=\{1, \cdots, n\}$ and edges $E$, with a weighted metric $\omega=(\omega_{ij}).$ It given by a $n$ by $n$  symmetric matrix with nonnegative entries $\omega_{ij}$ such that $\omega_{ij}>0$ if $(i,j) \in E.$ For simplicity, assume that the graph is connected, simple, with no self--loops or multiple edges. Let $\mathcal P(G)$ denote the probability simplex  
$$\Bigl\{\rho \in [0,1]^n\;| \; \sum_{i=1}^n \rho_i=1 \Bigr\},$$ 
the set  of probability measures on $V$. 
\noindent Any symmetric function $g:[0,\infty)^2 \rightarrow [0,\infty)$ induces an equivalence relation on $S^{n \times n},$ the set of $n$ by $n$  skew--symmetric matrix: if $\rho \in  \mathcal P(G)$,  $v, \tilde v \in S^{n \times n}$  are equivalent if \eqref{eq:equivalence-rel} holds. The quotient space $\mathbb H_\rho$ is  endowed with a metric tensor $(g_{ij}(\rho))_{ij}= (g(\rho_i, \rho_j))_{ij}$ which yields the inner product and norm in \eqref{eq:inner-product-and-norm}. The function $g$ is used to produce the underlying Hamiltonian $\mathcal H_g: \mathbb R^n \times \mathbb R^n \rightarrow \mathbb R$ given by   
\begin{equation}\label{eq:Hamiltonian-g}
\mathcal H_g(\rho, \phi)={1\over 4} \sum_{(i,j) \in E} \omega_{ij} g(\rho_i, \rho_j) (\phi_i-\phi_j)^2.
\end{equation}

We define the minimal action needed to connect  $\rho^0 \in \mathcal P(G)$ to $\rho^1 \in \mathcal P(G)$ to be 
\begin{equation}\label{eq:hamilton-min}
{1 \over 2}\mathcal W^2_g(\rho^0, \rho^1):= \inf\biggl\{ \int_0^1 \mathcal H_g(\rho, \phi) dt \; \Big| \; \dot \rho= \nabla_\phi \mathcal H_g(\rho, \phi),\; \rho(0)=\rho^0, \rho(1)=\rho^1\biggr\}.
\end{equation}
Making appropriate assumptions on $g$, this infimum will be shown to coincide with that in \eqref{2-W dist} and  $\mathcal W_g$ will be shown to be  a metric on $\mathcal P(G)$. Note that for a well--chosen sequence $(\phi_k)_k \subset \mathbb R^n$ whose norm tends to $\infty$, we may have that $\bigl(\mathcal H_g(\rho, \phi_k) \bigr)_k$ is identically null and so,  $\mathcal H_g(\rho, \cdot)$ is not coercive. This makes it a harder task to use direct methods of the calculus of variations to assert existence of a minimizer in \eqref{eq:hamilton-min}. To circumvent this obstacle, we instead use the equivalent formulation \eqref{2-W dist} and resort to identifying a dual to \eqref{2-W dist}. 

Any minimizer  $(\rho, \phi)$ of \eqref{eq:hamilton-min} such that $\rho, \phi \in W^{1,2}(0,1; \mathbb R^n)$ and the range of $\rho$ does not intersect the boundary of  $\mathcal P(G)$ satisfies the Euler--Lagrange equation 
\begin{equation}\label{eq:hamilton-min2}
\dot \rho= \nabla_\phi \mathcal H_g(\rho, \phi),\quad \dot \phi= -\nabla_\rho \mathcal H_g(\rho, \phi).
\end{equation}
Using the notation of graph divergence (cf. Section \ref{section2}), this  Hamiltonian system which reads off 
\begin{equation}\label{eq:hamilton-min2c}
\dot \rho_i + \sum_{j \in N(i)} \sqrt{\omega_{ij}}(\phi_j-\phi_i)g(\rho_i,\rho_j)=0,\quad \dot \phi_i+{1\over 2} \sum_{j \in N(i)} \omega_{ij} \partial_1 g(\rho_i, \rho_j)(\phi_i-\phi_j)^2=0.
\end{equation}
Here, $N(i) := \{j \in V \;  |  \; (i,j) \in E\}$ denotes the neighborhood of a vertex $i \in V$ and $\partial_1 g$ denotes the partial derivative of $g$ with respect to its first variable. 

Hardly enough, even if $\rho^0, \rho^1$ are chosen in the interior of $\mathcal P(G)$, a minimizer $(\rho, \phi)$ of \eqref{eq:hamilton-min}  may be such that the range of $\rho$ intersects the boundary of  $\mathcal P(G)$ unless (cf. \cite{Maas})
\begin{equation}\label{eq:hamilton-min2b}
C_g:= \int_0^1 {dr \over \sqrt{g(r, 1-r)}} =\infty. 
\end{equation}
The  condition \eqref{eq:hamilton-min2b}, precisely forces $\mathcal W_g$ to assume infinite values and so, it cannot be a metric on the whole set $\mathcal P(G)$. When $C_g<\infty$, we rather endeavor to identify the appropriate substitute of \eqref{eq:hamilton-min2c}, by characterizing minimizers of \eqref{2-W dist}, even when $\rho((0,1))$ intersects the boundary of  $\mathcal P(G).$ 

We define a Poincar\'e function $\gamma_P: \mathcal P(G) \rightarrow \mathbb R$. It is a concave function which is strictly positive in the interior of  $\mathcal P(G)$ but may remain positive on a subset of the boundary of  $\mathcal P(G)$. When $\gamma_P(\rho^0), \gamma_P(\rho^1)>0$ we show that $(\rho, v)$ minimizes \eqref{2-W dist} if and only if there exists $\lambda \in {\rm BV}_{loc}\bigl(0,1; \mathbb R^n \bigr)$ such that $-\dot \lambda_i$ is a Borel regular measure such that 
\[
g(\rho_i, \rho_j) \Bigl[ v_{ij}- \sqrt{\omega_{ij}}(\lambda_i-\lambda_j) \Bigr]=0 \quad \forall (i, j) \in E, 
\]
\begin{equation}\label{eq:hamilton-min2e}
0= H(\dot \lambda^{abs}, \nabla_G \lambda)=  (\dot \lambda^{abs}, \rho)+ {1\over 2} \|\nabla_G \lambda\|^2_\rho \quad \mathcal{L}^1\,\,\text{a.e.} 
\end{equation}
and 
\begin{equation}\label{eq:hamilton-min2f}
0=H_0\Bigl({d \lambda^{sing} \over d\nu}\Bigr)=\Bigl({d \lambda^{sing} \over d\nu}, \rho\Bigr) \quad \nu\,\, \text{a.e.}.
\end{equation}
Here, $\dot \lambda^{abs}$ is the absolutely continuous part of $\dot \lambda$, $\lambda^{sing}$ is the singular part of $\dot \lambda$, $\nu$ is any non-negative measure such that $\nu$ and $\mathcal L^1|_{(0,1)}$ are mutually singular and $|\lambda^{sing}|<<\nu.$ We have defined $H$ as in \eqref{eq:defn-of-H} and set 
\[
H_0(a)= \max_{1\leq i\leq n} a_i \qquad \forall a \in \mathbb R^n.
\]
There is a relation between $H_0$ and the recession function of $H$ since  
\[
\lim_{l \rightarrow \infty} {H(la, lb) \over l}= \begin{cases}
\hfill H_0(a) & \text{if} \quad b=0\\
\infty & \text{if} \quad b\not =0.
\end{cases}
\]
What seems surprising at a first glance is that,  even when $\dot \lambda$ has no singular part, the expression in \eqref{eq:hamilton-min2e} is still not linear  in $\dot \lambda$.  This means the geodesics of minimal length are characterized Hamilton--Jacobi equations in the form $0= H(\dot \lambda^{abs}, \nabla_G \lambda),$ with a non--linear dependence in $\dot \lambda^{abs}.$ This is in contrast with what happened in the continuum setting, where there, geodesics of minimal length are characterized by Hamilton--Jacobi equations in the form $\partial_t u+H_*(\nabla u)=0$, hence linear in $\partial_t u$. We pause here to draw the attention of the reader to \cite{Shu17} which proposes a class of Hamilton--Jacobi equations which can hardly be compared with Remark \ref{re:necessary-sufficient} (ii).

A comparison between our work and the innovative work \cite{Maas} by J. Maas, becomes at this point unavoidable. There, the author  considers an  irreducible Markov kernel $(K_{ij})_{ij}$ with a finite right invariant measure. Our hypothesis that $(\omega_{ij})_{ij}:=(K_{ij} \pi_i)_{ij}$ be symmetric, is equivalent to the requirement in  \cite{Maas} that $K$ be reversible.  When $C_g=\infty$, \cite{Maas} gave a remarkable characterization of the pairs for which $\mathcal W_g(\rho^0, \rho^1)<\infty$. The necessary and sufficient condition is that both $\rho^0$ and $\rho^1$ must have the same $g$--connected components (see Section \ref{section2} for the definition of $g$--connected components). As a consequence, if $(\rho, v)$ is a minimizer in \eqref{2-W dist} then the $g$--connected components of $\rho(t)$ are independent of $t\in (0,1)$ and they coincide with those of $\rho^0.$ The search of paths of minimal actions in \eqref{2-W dist} reduces then to a finite collection of searches of paths of minimal actions which are known to be entirely contained in the interior of simplexes. In this case, the Euler--Lagrange equations are obtained by standard arguments. 

In this manuscript, we assume that $C_g<\infty$ and so, our study of geodesics of minimal norms complements that  in \cite{Maas}. We further assume that $g$ is concave, $1$--homogeneous, positive in $(0,\infty)^2$ and $C^\infty$ in this open set. These assumptions, while facilitating our study, still encompasse a large number of metrics, useful in modeling and computations of 2--Wasserstein metric. The class of functions $g$ we choose are motivated by  studies  \cite{chow2012, li_2017, li_SE, Maas} which recently appeared in the literature. 

The study in this manuscript will be more than a disappointment if the set of geodesics starting and ending in the interior of $\mathcal P(G)$ would never intersect the boundary of $\mathcal P(G)$. Unlike the study in  \cite{Maas}, Proposition \ref{pr:nonempty} supports the fact that when $C_g<\infty$ then the set of such geodesics is not void. Another feature of the condition  $C_g<\infty$, is that if $\rho:[0,1] \rightarrow \mathcal P(G) \setminus \mathcal P_0(G)$ is a geodesic of minimal length, then the $g$--connected components of $\rho(t)$ needs not to be time independent (cf. Proposition \ref{pr:example-bdry})  unlike the case when $C_g=\infty$ \cite{Maas}. One could  combine Propositions \ref{pr:example-bdry} and \ref{pr:nonempty} to construct more intricate geodesics which intersect the boundary of $\mathcal P(G)$.

The manuscript is organized as follows. In Section \ref{section2}, we introduce the notation used in the manuscript. Section \ref{section3} contains preliminary remarks. For instance there, we comment on the sufficient condition for $\mathcal W_g$ to assume only finite values. In Section \ref{section4} we show existence of geodesics of minimal norms. Sections  \ref{section5} and \ref{section6} contains ingredients we later use in  Section \ref{section7} to  characterize the geodesics of minimal path through a dual formulation.

%
%
\section{Notation}\label{section2}
We denote the one--dimensional Lebesgue measure by $\mathcal{L}^1$ and denote  the set of skew--symmetric $n \times n$ matrices as $S^{n \times n}$. Let $G=(V, E, \omega)$ denote an undirected graph  of vertices $V=\{1, \cdots, n\}$ and edges $E$, with a weighted metric $\omega=(\omega_{ij})$ given by a $n$ by $n$  symmetric matrix with nonnegative entries $\omega_{ij}$ and such that $\omega_{ij}>0$ if $(i,j) \in E.$ For simplicity, assume that the graph is connected and is simple, with no self--loops or multiple edges.  

\hfill\break
{\bf Functions on a graph.} It is customary to identify a function $\phi: V \rightarrow \mathbb R$ with a vector $\phi= (\phi_i)_{i=1}^n \subset \mathbb R^n$. We use the standard {\em inner product} on $\mathbb R^n$: 
\[
(\phi,\tilde \phi ):=\sum_{i=1}^n \phi_i \tilde \phi_i, \quad \forall \; \phi, \tilde \phi \in \mathbb R^n.  \; 
\]
\hfill\break
{\bf Vector fields and gradient operator.}  A vector field $m$ on $G$ is a  {\em skew-symmetric matrix} on the edges set $E$, denoted by $m$:  
\[
m:=(m_{ij})_{(i,j)\in E},\quad\textrm{with} \quad m_{ij}=-m_{ji}.
\]
Special  elements of $S^{n \times n}$ are the so--called {\rm potential vector fields}  which are {\em discrete gradients} of functions $\phi$ on $V$, denoted $\nabla_G \phi$  and defined as  
\[
\nabla_G\phi:=\sqrt{\omega_{ij}}(\phi_i-\phi_j)_{(i,j)\in E}.
\] \\
{\bf The range and kernel of the gradient operator.} We denote by $R(\nabla_G)$ the range of $\nabla_G$ and by ${\bf 1} \in \mathbb R^n$ the vector whose entries are all equal to $1$. Since $G$ is connected, the kernel of $\nabla_G$ is the one dimensional space spanned by ${\bf 1}.$ The orthogonal in $\mathbb R^n$ of the latter space is $\kerGp,$ the set of $h \in \mathbb R^n$ such that $\sum_{i=1}^n h_i=0.$ \\ \\
{\bf $G$--Divergence of vector field.} The {\em divergence} operator associates to any vector field $m$ on $G$ a function on $V$ defined by 
\[
\nabla_G \cdot (m)=\textrm{div}_G(m):=\Bigl(\sum_{j\in N(i)}\ow m_{ji} \Bigr)_{i=1}^n.
\]\\
{\bf Set of probability measures and its boundary.} We identify $\mathcal{P}(G)$, the set of probability measures on $V,$  with a simplex as follows  
\[
\mathcal{P}(G)=\Bigl\{\rho=( \rho_i)_{i=1}^n \subset [0,1]^n\;\; \Big | \;\; \sum_{i=1}^n \rho_i=1 \Bigr\}.
\]
Let  $\mathcal{P}_0(G):=\mathcal{P}(G) \cap (0,1)^n$ denote the interior of $\mathcal{P}(G)$. The boundary of $\mathcal{P}(G)$  is $\mathcal{P}(G) \setminus \mathcal{P}_0(G).$  
\\
\hfill\break
{\bf The set $\mathcal C(\rho^0, \rho^1)$ of paths connections probability measures.}
Given $\rho^0, \rho^1 \in \mathcal{P}(G),$ we denote as $\mathcal C(\rho^0, \rho^1)$ the set of pairs $(\rho, m)$ such that  $\rho: [0,1] \rightarrow  \mathcal{P}(G)$  
\[ \rho_i\in H^1(0,1), \; m_{ij}\in L^{2}(0,1) \quad \forall  (i,j) \in E, \quad   (\rho(0), \rho(1))=(\rho^0, \rho^1)\]
and 
\begin{equation}\label{3_newb}
\dot \rho_i +\sum_{j\in N(i)}\ow m_{ji}=0, \quad \text{in the weak sense on} \quad (0,1).
\end{equation}\\
Throughout this manuscript  $g:[0,\infty)\times[0,\infty)\to\mathbb R_+$ satisfies the following assumptions:
\begin{itemize}
\item [(H-i)] $g$ is continuous on $[0,\infty)\times[0,\infty)$ and is of class $C^{\infty}$ on $(0,\infty)\times(0,\infty)$;
\item [(H-ii)] $g(r,s)=g(s,r)$ for any $s,r\in\mathbb R_+$;
\item [(H-iii)] $g(r,s)>0$ for any $r,s\in(0,\infty)$;
\item [(H-iv)] $g(\lambda r,\lambda s)=\lambda g(r,s)$ for any $\lambda,s,r \in(0,\infty)$;
\item [(H-v)] $g$ is concave.
\end{itemize}
We extend $g$ by setting its value to be $-\infty$ outside  $[0,\infty)^2$, to obtain a function on $\mathbb R^2$ which we still denote $g.$ Observe that the extension is concave and upper semicontinuous. We define 
\[
g_{ij}(\rho)= g(\rho_i, \rho_j) \quad \forall \; \rho \in \mathbb R^n, \quad \forall \; i, j \in V.
\] \\
{\bf A constant depending solely of $g$.} Since $g$ is continuous on the compact set $[0,1]^2$,  
\[
\epsilon_0(g):= \sup_{r, s>0} g\Bigl({r \over r+s},  {s \over r+s}\Bigr)
\]
is a finite number.\\ \\
{\bf The Hilbert spaces $\mathbb H_\rho.$} If $\rho \in \mathcal{P}(G)$, we say that $v, \tilde v \in S^{n \times n}$ are $\rho$--equivalent if 
\begin{equation}\label{eq:equivalence-rel} 
(v_{ij}-\tilde v_{ij})g_{ij}(\rho)=0 \qquad \forall (i,j) \in E,
\end{equation} 
which means $v_{ij}=\tilde v_{ij}$ whenever $g_{ij}(\rho)>0.$ We denote by $\mathbb H_\rho$ the set of class of equivalence. This is a Hilbert space when endowed with the discrete {\em inner product} and the {\em discrete norm}  
\begin{equation}\label{eq:inner-product-and-norm}
(v,\tilde v )_ \rho:=\frac{1}{2}\sum_{(i,j)\in E} v_{ij}\tilde v_{ij}g_{ij}(\rho), \quad \|v\|_\rho=\sqrt{(v,v )_ \rho } \quad \forall \; v, \tilde v \in S^{n \times n}.  
\end{equation}
Here the coefficient $1/2$ accounts for the fact that whenever $(i,j) \in E$ then $(j,i) \in E$. Similarly, if $m, \tilde m \in S^{n \times n}$ we set 
\[
(m,\tilde m ):=\frac{1}{2}\sum_{(i,j)\in E} m_{ij} \tilde m_{ij}, \quad \|m\|^2=(m,m ).  
\] \\
\hfill\break
{\bf The tangent spaces and the projection operator $\pi_{\rho}$.} We denote as $T_\rho \mathcal P(G)$ the closure of the range of $\nabla_G$ in $\mathbb H_\rho.$ We refer to $T_\rho \mathcal P(G)$ as the tangent space to $\mathcal P(G)$ at $\rho.$ Given $v \in  \mathbb H_\rho$ there exists a unique $\pi_\rho(v) \in  T_\rho \mathcal P(G)$ that minimizes $\|v -\cdot\|_\rho$ over $T_\rho \mathcal P(G)$. It is characterized by the property 
\begin{equation}\label{eq:oct22.2017.2}
(v- \pi_\rho(v), w)_\rho= 0 \qquad \forall w \in T_\rho \mathcal P(G).
\end{equation}
\hfill\break 
{\bf The divergence operator.} The operator $\nabla_G: \mathbb R^n \rightarrow \mathbb H_\rho$ admits an adjoint  $-\divrho: \mathbb H_\rho \rightarrow \mathbb R^n$ given by  
\[
\divrho(v)=\biggl(\sum_{j\in N(i)}\ow v_{ji}g_{ij}(\rho)\biggr)_{i=1}^n \quad \forall \; v \in S^{n \times n}.
\] 
We call $\divrho$ the divergence operator. Note the integration by parts  formula:
\begin{equation}\label{eq:july11.2017.1}
(\nabla_G \phi, v)_\rho= - (\phi, \divrho(v)).
\end{equation}
Let  $\mathcal H_g$ be the Hamiltonian defined in \eqref{eq:Hamiltonian-g}. Observe that 
\begin{equation}\label{eq:dec06.2017.1}
\divrho(\nabla_G \phi)=- \nabla_\phi \mathcal H_g(\rho, \phi).
\end{equation}
\hfill\break
{\bf The Monge--Kantorovich metric on $G$.} The square of $2$--Monge--Kantorovich metric which measures the square distance between $\rho^0 \in \mathcal{P}(G) $ and $\rho^1 \in \mathcal{P}(G)$ is 
\begin{equation}\label{2-W dist}
\mathcal W_g^2(\rho^0, \rho^1):= \inf_{(\rho,v)} \Bigl\{~\int_0^1 (v, v)_\rho dt\;\; \Big|\,\,\dot \rho +\textrm{div}_\rho (v)=0,\,\, \rho(0)=\rho^0,\,\,\rho(1)=\rho^1\Bigr\}.
\end{equation} 
Here the infimum is performed over the set of pairs $(\rho, v)$ such that $\rho \in H^1\left(0,1; \mathbb R^n\right)$, $v:[0,1] \rightarrow S^{n \times n}$ is measurable. 

\hfill\break
{\bf Connected components.} Let $\rho \in \mathcal P(G).$ We say that $i, j \in V$ are  $g$--connected  if there are integers $i_1,i_2, \cdots, i_k\in V$ such that $i_1=i$, $i_k=j$, $(i_l,i_{l+1})\in E$ for $l=1,\cdots,k-1$ and
\[
g(\rho_{i_1}, \rho_{i_2}) \cdots g(\rho_{i_{k-1}}, \rho_{i_k}) >0.
\]
The largest $g$--connected set containing $i$ is called the $g$--connected component of $i$. The $g$--connected components of $\rho$ form a partition of a subset of $V.$ 

\hfill\break
{\bf Poincar\'e functions on graphs.} We define the Poincar\'e function $\gamma_P$ on $G$ as 
\[
\gamma_P(\rho)=\inf_{\beta}\left\{\frac{1}{2}\sum_{(i,j)\in E}g_{ij}(\rho)\oww(\beta_i-\beta_j)^2\;\; \Big| \;\;\sum_{i=1}^n\beta_i=0,\,\,\sum_{i=1}^n\beta_i^2=1\right\} \qquad \forall \rho \in \mathcal P(G).
\]
\hfill\break
{\bf Action.} Consider the lower semicontinuous convex function $f: \mathbb R^2 \rightarrow [0,\infty]$ defined as 
\begin{equation}
f(t,s)=\left\{
\begin{array}
[c]{cl}
{s^2 \over t} & \;\; \text{if}\;\; t>0 \smallskip\\
0 &\;\; \text{if}\;\;s=t=0\smallskip\\
\infty &\;\; \text{otherwise}.
\end{array}
\right.  \label{eq:example-of-c1}%
\end{equation}
Observe that if $t \geq 0$ and $\mu \in \mathbb R$ then 
\begin{equation}\label{eq:partial-Legendre}
2\mu s< f(t, s)+\mu^2 t
\end{equation} 
unless $t \mu=s$ in which case equality holds. 

\hfill\break
For $\rho \in \mathbb R^n$ and $m \in S^{n\times n}$, we define 
\[
F(\rho, m)=\frac{1}{2}\sum_{(i,j) \in E} f\bigl(g_{ij}(\rho), m_{ij}\bigr).
\]
If $\rho \in L^{2}(0,1; \mathbb R^n)$ and $m \in L^2(0,1; S^{n \times n})$ we define the action 
\[
\mathcal A(\rho, m)= {1\over 2} \int_0^1 F(\rho, m) dt.
\]
Let $H: \mathbb R^{n}\times  S^{n\times n} \rightarrow \mathbb R$ denote the Hamiltonian defined as 
\begin{equation}\label{eq:defn-of-H}
H(a,b):=\sup_{\rho\in\mathcal{P}(G)}\Bigl\{\left( a,\rho\right)+\frac{1}{2}\|b\|^2_\rho \Bigr\}\quad \forall (a,b)\in \mathbb R^{n}\times  S^{n\times n}.
\end{equation}

In the remaining of the manuscript, unless the contrary is explicitly stated, we assume that 
\begin{equation}\label{eq:finite energy}
C_g:=\int_0^1\frac{dr}{\sqrt{g(r,1-r)}}<+\infty.
\end{equation}

\begin{example} Examples satisfying {\rm(H-i)-(H-v)} and \eqref{eq:finite energy} include $g(r,s)=\frac{r+s}{2}$. Other examples which appeared in \cite{Maas}) are 
\[
g(r,s)=\int_0^1r^{1-t}s^tdt=\left\{
\begin{array}
[c]{cl}
{\frac{r-s}{\log r-\log s}} &\;\; \text{if}\;\;r\not= s \smallskip\\
0 & \;\; \text{if}\;\;r=0\,\,\text{or}\,\, s=0 \smallskip\\
r & \;\; \text{if}\;\; r=s,
\end{array}
\right. 
\]
and 
\[ 
g(r,s)=\left\{
\begin{array}
[c]{cl}
{0} &\;\; \text{if}\;\;r=0\,\,\text{or}\,\, s=0 \smallskip\\
\frac{1}{\frac{1}{r}+\frac{1}{s}} & \;\; \text{otherwise}.
\end{array}
\right.  
\] 
\end{example}

%
%
%
\section{Preliminaries} \label{section3}
In this section, we use the same notation as in Section \ref{section2} and assume  \eqref{eq:finite energy} hold.
\begin{lemma}\label{lem:concave}
The Poincar\'e function $\gamma_P:\mathcal{P}(G)\to\mathbb R$ is concave.
\end{lemma}
\begin{proof} Note $\gamma_P$ is obtained by taking the infimum of concave functions of $\rho$.
\end{proof}

\begin{lemma}\label{le:Poincare-inequality}
If $\rho\in\mathcal{P}(G)$, $\lambda \in \mathbb R^n,$ $na= \sum_{j=1}^{n}\lambda_j$ and  $\tilde\lambda_i:=\lambda_i-a$ then  
\[
\|\nabla_G \tilde\lambda\|_{\rho}^2\geq \gamma_P(\rho)\|\tilde\lambda\|^2.
\]
\end{lemma}
\begin{proof}
Set
\begin{equation*}
\beta_i:=\frac{\tilde\lambda_i}{\sqrt{\sum_{j=1}^n\tilde\lambda_j^2}}.
\end{equation*}
Then
\begin{equation*}
\sum_{i=1}^{n}\beta_i=\frac{\sum_{i=1}^{n}\tilde\lambda_i}{\sqrt{\sum_{j=1}^{n}\tilde\lambda_j^2}}=0\,\,\text{and}\,\,\sum_{i=1}^{n}\beta_i^2=\frac{\sum_{i=1}^{n}\tilde\lambda_i^2}{\sum_{j=1}^{n}\tilde\lambda_j^2}=1.
\end{equation*} 
The desired inequality follows from the definition of $\gamma_P$.\end{proof}

\begin{remark}\label{rem:lsc-of-F} Suppose $\rho\in\mathcal{P}(G)$ has only one $g$--connected component which is the whole set $V$.  Then the range of $\nabla_G$ is a closed subset of $ \mathbb H_\rho$ and so,   it is $T_\rho \mathcal P(G)$.
\end{remark}
\proof{} 
Suppose $(\phi^k)_k \subset \mathbb R^n$ is such that $(\nabla_G \phi^k)_k $ converges to $v$ in $\mathbb H_\rho$. We are to show that $v \in R(\nabla_G)$. For any $e\in V\setminus\{1\}$, there exists $e_1,\cdots,e_l\in V$ such that $e_1=1$, $e_l=e$, $(e_j,e_{j+1})\in E$ and $g_{e_je_{j+1}}(\rho)>0$ for any $j\in\{1,\cdots,l-1\}$. We have for any $j\in\{1,\cdots,l-1\}$
\begin{equation}\label{eq:oct22.2017.1}
\lim_{k \rightarrow \infty} \bigl( \sqrt{\omega_{e_je_{j+1}}}{(\phi^k_{e_j}-\phi^k_{e_{j+1}})}-v_{e_j e_{j+1}} \bigr)^2 g_{e_je_{j+1}}(\rho)=0.
\end{equation}
Replacing $\varphi_{e_j}^k$ by $\varphi_{e_j}^k-\varphi_1^k$ if necessary, we may assume without loss of generality that $\varphi^k_{1}=0$.

Setting $j=1$ in \eqref{eq:oct22.2017.1}, we obtain that $(\phi^k_{e_2})_k$ converges to $\frac{-v_{e_1 e_2}}{\sqrt{\omega_{e_1e_2}}}.$ Setting $\phi_{e_1}=0$, $\phi_{e_2}=\frac{-v_{e_1 e_2}}{\sqrt{\omega_{e_1e_2}}}$ we inductively  obtain 
\[
\sqrt{\omega_{e_{j}e_{j+1}}}\phi_{e_{j+1}}:= \lim_{k \rightarrow \infty} \sqrt{\omega_{e_{j}e_{j+1}}}\phi^k_{e_{j+1}}= \lim_{k \rightarrow \infty} \sqrt{\omega_{e_{j}e_{j+1}}}\phi^k_{e_{j}}-v_{e_j e_{j+1}} = \sqrt{\omega_{e_{j}e_{j+1}}}\phi_{e_{j}}-v_{e_j e_{j+1}}
\]
for any $j\in \{2, \cdots, l-1 \}$. This is sufficient to verify $v=\nabla_G \phi.$ \endproof

\begin{lemma}\label{le:complete-path1}
Assume that $\rho\in\mathcal{P}(G)$. Then  $\rho$ has only one $g$--connected component which is the whole set $V$ iff $\gamma_P(\rho)>0$.
\end{lemma}
\begin{proof}
Assume that  $\rho$ has only one $g$--connected component which is the whole set $V$. Then for any $e\in V\setminus\{1\}$
there exists $e_1,\cdots,e_l\in V$ such that $e_1=1$, $e_l=e$, $(e_j,e_{j+1})\in E$ and $g_{e_je_{j+1}}(\rho)>0$ for any $j\in\{1,\cdots,l-1\}$. Suppose that $\gamma_P(\rho)=0$, i.e. there exists $\beta\in \mathbb R^{n}$ such that $\sum_{i=1}^{n}\beta_i=0$, $\sum_{i=1}^{n}\beta_i^2=1$ and
\begin{equation}\label{eq:=0}
\sum_{(i,j)\in E}g_{ij}(\rho)\oww(\beta_i-\beta_j)^2=0.
\end{equation}
\eqref{eq:=0} implies that
\begin{equation*}
0=\sum_{(i,j)\in E}g_{ij}(\rho)\oww(\beta_i-\beta_j)^2\geq\sum_{j=1}^{l-1}g_{e_je_{j+1}}(\rho)\omega_{e_je_{j+1}}(\beta_{e_j}-\beta_{e_{j+1}})^2=0
\end{equation*}
and, thus, $\beta_{e_1}=\beta_{e_2}=\cdots=\beta_{e_l}$. Since $e$ is arbitrary, we have $\beta_{1}=\beta_{2}=\cdots=\beta_{n}=0$. This is in contradiction with the fact that $\sum_{i=1}^{n}\beta_i^2=1$.

Suppose that $\gamma_P(\rho)>0$. We want to prove that  $\rho$ has only one $g$--connected component which is the whole set $V$. If not, there exist $i_1,j_1\in V$ such that $i_1$ and $j_1$ are not in the same $g$--connected component. Let $V_{i_1}=\{i_1,i_2,\cdots,i_k\}$ and $V_{j_1}=\{j_1,j_2,\cdots,j_{\tilde k}\}$ be the $g$--connected components of $i_1$ and $j_1$ respectively. Set $\beta_i=0$ whenever $i\in V\setminus (V_1\cup V_2)$. Otherwise, set 
\begin{equation*}
\beta_{i_1}=\beta_{i_2}=\cdots=\beta_{i_k}=\sqrt{\frac{\tilde k}{k(k+\tilde k)}}\,\,\text{and}\,\,\beta_{j_1}=\beta_{j_2}=\cdots=\beta_{j_{\tilde k}}=-\sqrt{\frac{k}{\tilde k(k+\tilde k)}}.
\end{equation*}
Then we have $\sum_{i=1}^{n}\beta_i=0$, $\sum_{i=1}^{n}\beta_i^2=1$ and \eqref{eq:=0} holds. This is at variance with the fact that $\gamma_P(\rho)>0$.
\end{proof}

\begin{remark}\label{re:Legendre-trans} We assert the following

\begin{enumerate}
\item[(i)] The function $F$ is  a convex and lower semicontinuous.
\item[(ii)] Suppose $m, b \in S^{n \times n}$ and $\rho \in \mathcal P(G)$ are such that $m_{ij}=0$ whenever $g(\rho_i, \rho_j)=0.$   Then   
\[
F(\rho, m)+ \|b\|_\rho^2 > 2 (m, b),
\]
\noindent unless $m_{ij}=g(\rho_i, \rho_j) b_{ij}$ for all $(i, j) \in E,$ in which case equality holds. 
\end{enumerate}
\end{remark} 
\proof{} (i) Since $g$ is concave, $f$ is convex and $f(\cdot,s)$ is monotone non--increasing, $(\rho, m) \rightarrow f\bigl(g_{ij}(\rho), m_{ij}\bigr)$ is convex and so, function $F$ is convex. One checks that $(\rho, m) \rightarrow f\bigl(g_{ij}(\rho),m_{ij}\bigr)$ is lower semicontinuous. Thus,  $F$ is  a convex and lower semicontinuous as a sum of convex, lower semicontinuous functions. 

(ii) is a direct consequence of \eqref{eq:partial-Legendre}. \endproof 

\begin{lemma}\label{le:Legendre-trans}  let $\mathcal H_g$ be as in \eqref{eq:Hamiltonian-g} and let $\partial_i g$ denote derivative of $g$ with respect to the $i$--th variable for $i=1,2.$
\begin{enumerate}
\item[(i)]   If $r, s>0$ then   
\[
\partial_1 g(r, s)= \partial_2 g(s, r) \quad \text{and} \quad \bigl(\nabla g(r, s) , (r,s)\bigr)=g(r, s).
\]
\item[(ii)] For any $\rho \in [0,\infty)^n$ and $\phi \in \mathbb R^n$ we have 
\[ 
 \bigl(\nabla_\phi \mathcal H_g(\rho, \phi) , \phi\bigr)= 2\mathcal H_g(\rho, \phi).
\] 
\item[(iii)] For any $\rho \in (0,\infty)^n$ and $\phi \in \mathbb R^n$ we have 
\[ 
 \bigl(\nabla_\rho \mathcal H_g(\rho, \phi) , \rho\bigr)= \mathcal H_g(\rho, \phi).
\] 
\end{enumerate} 
\end{lemma} 
\proof{} Recall that $g$ has been extended to an upper semicontinuous on $\mathbb R^2$, which we still denote as $g$. 

(i) Since $g(r, s)=g(s, r)$, differentiating, we obtain the first identity in (i). Let $G^*$ denote the Legendre transform of the  convex, degree $1$--homogeneous function $G=-g$. If $\alpha, \beta \in \mathbb R$ then $G^*(\alpha, \beta)=\infty,$ unless $\alpha \bar r +\beta \bar s \leq G(\bar r, \bar s)$ for all $\bar r, \bar s \in \mathbb R$, in which case $G^*(\alpha, \beta)=0.$ If $r, s>0$ then $G$ is differentiable at $(r, s)$ and so, setting $(\alpha, \beta)=\nabla G(r, s)$ we have 
\[
\bigl(\nabla G(r, s) ,(r, s)\bigr)= \alpha r +\beta s=G(r, s)+G^*(\alpha, \beta)=G(r, s).
\] 
This completes the verification of (i).

(ii) We have 
\[
{\partial \mathcal H_g \over \partial \phi_i}(\rho, \phi)=\sum_{j\in N(i)} \omega_{ij} g(\rho_i, \rho_j) (\phi_i-\phi_j).
\] 
We use this to verify that (ii) holds.

(iii) We use the first identity in (i) to infer 
\[
{\partial \mathcal H_g \over \partial \rho_i}(\rho, \phi)= {1 \over 2} \sum_{j\in N(i)} \omega_{ij} \partial_1 g(\rho_i, \rho_j) (\phi_i-\phi_j)^2.
\] 
This, together with the second identity in (i) complete the verification of (iii). \endproof

\begin{proposition}\label{pro:finite-energy-n-vertices} For any $\rho^0, \rho^1\in\mathcal{P}(G)$ there is a path $(\rho,m)\in \mathcal C(\rho^0, \rho^1)$ such that $\mathcal{A}(\rho,m)<+\infty$. In other words, we have a feasible path in $\mathcal C(\rho^0, \rho^1)$.  
\end{proposition}
\proof{} Let $C_g$ be as defined in \eqref{eq:finite energy}. Changing variables, we infer 
\[
C_g= \sqrt 2 \int_0^1\frac{dr}{\sqrt{g(1+r,1-r)}}.
\]
This new formulation of $C_g$ allows us to attribute this Proposition to \cite{Maas} even if one may think his setting and ours seem to be a variant of each other.  For completeness we lay down the main arguments supporting our statement. 

Let $G_2=(V_2,E_2,\omega_2)$ be a graph of vertices $V_2=\{1,2\}$, edges $E_2=\{(1,2), (2,1)\}$, endow with the weight $\omega_{12}=\omega_{21}>0$. Let $\rho^0, \rho^1 \in \mathcal P(G_2)$. To avoid trivialities, assume $\rho_1^0\not =\rho_1^1.$ Without loss of generality, assume  $\rho_1^0< \rho_1^1.$ Note the  strictly increasing function 
\[
\tau \rightarrow G(\tau):=\int_0^\tau \frac{dr}{\sqrt{g(r,1-r)}} 
\]
has an inverse function $G^{-1}$ which is differentiable.  Set 
\[
C=G(\rho_1^1)-G(\rho_1^0), \quad \rho_1(t)=G^{-1}\left(G(\rho_1^0)+Ct\right), \quad   \rho_2(t)=1-  \rho_1(t).
\]
Define $m_{21}=-m_{12}$ through the identity $$\sqrt{\omega_{12}}m_{21}=-\dot \rho_1.$$ 
Observe that the path $t \rightarrow (\rho_1(t), 1- \rho_1(t))$ connects $(\rho^0_1, \rho^0_2)$ to $(\rho^1_1, \rho^1_2)$ and 
\[ 
\omega_{12}m_{12}^2=\omega_{12}m_{21}^2= \bigl({d\rho_1 \over dt} \bigr)^2= C^2 g(\rho_1, 1-\rho_1) \in L^1(0,1).
\] 
By definition $(\rho, m)$ satisfies  \eqref{3_newb}. Check that  $2\omega_{12} \mathcal F(\rho, m)= C^2.$ This covers the case $n=2.$ 

When $n>2$, if $\rho\in\mathcal{P}(G)$ is such that there is a feasible path in $\mathcal C(\rho^0, \rho)$ and a feasible path in $\mathcal C(\rho^1, \rho)$ then by concatenation, there is a feasible path in $\mathcal C(\rho^0, \rho^1).$ This means we may assume without loss of generality that $\rho^1=(0,\cdots, 0, 1).$ Let 
$$
V[{\rho}]:=\bigl\{i\in \{1,\cdots n-1\}\;\; | \;\;\rho_i>0\bigr\}.
$$ 
If $V[{\rho^0}]=\emptyset$, then $\rho^0=\rho^1$ and so, $(\rho, m) \equiv (\rho^0, 0)$ is a feasible path in $\mathcal C(\rho^0, \rho^1)$.  Assuming that $V[{\rho^0}]\not=\emptyset$,
one iteratively construct a finite sequence $\tilde \rho^0, \cdots, \tilde \rho^{l_0}$ in $\mathcal{P}(G)$ satisfying the following properties: 

(i)  $\tilde \rho^0=\rho^0$ and $\tilde \rho^{l_0}=\rho^1$; 

(ii) the cardinality of $V[\tilde \rho^l]$ is strictly smaller than that of $V[\tilde \rho^{l-1}]$ whenever $l\leq l_0;$ 

(iii)  there is  a feasible path in $\mathcal C(\tilde \rho^{l-1}, \tilde \rho^l).$\endproof

The following example will be useful in the next proposition:  
\begin{equation}\label{eq:dec08.2017.2}
g(r,s)= \left\{
\begin{array}
[c]{cl}
{r+s \over 2} &\;\; \text{if}\;\;r, s \geq 0 \smallskip\\
-\infty & \;\; \text{otherwise.}
\end{array}
\right. 
\end{equation}
We will later use the set 
\begin{equation}\label{eq:dec08.2017.1}
Q:=\{(r_1, r_3) \in (0, 1)^2 \; | \; r_1+r_3 <1 \}.
\end{equation}
\begin{proposition}\label{pr:example-bdry} Let $V=\{1, 2, 3\}$, $E=\{(1, 2), (2, 3),(1,3)\}$ and let $\omega$ denote a $3 \times 3$ symmetric matrix such that $\omega_{12}=\omega_{23}=1$ and  $\omega_{13}=0.$ Let $G$ denote the weighted graph $(V, E, \omega).$ Let $g$ be as in \eqref{eq:dec08.2017.2}.  Let $\rho^0=(0, 0, 1)$ and $\rho^1 =(0, \; 1/2,  \; 1/2)$ so that $\rho^0$ and $\rho^1$ lie on the boundary of $\mathcal P(G)$. Observe $\rho^0$ has only one $g$--connected component which is $\{2, 3\}$ and $\rho^1$   has only one $g$--connected component which is $\{1,2, 3\}.$ We claim that any geodesic of minimal norm $\rho:[0,1] \rightarrow  \mathcal P(G)$, connecting $\rho^0$ to $\rho^1$ lies in the boundary of $\mathcal P(G).$ Furthermore, the $g$--connected components of $\rho(t)$ are not constant in $t$.
\end{proposition}
\proof{}  By Theorem \ref{th:exist} there is $(\rho,m)$ that minimizes $\mathcal A$ over $\mathcal C(\rho^0, \rho^1).$  

In order to show  the range of $\rho$ lies in the boundary of $\mathcal P(G)$, it suffices to show that $\rho_1(t) \equiv 0.$ To achieve that goal, it suffices to show that for any $(\rho, m) \in \mathcal C(\rho^0, \rho^1)$ such that $\rho_1 \not \equiv 0,$ we can construct $(\bar \rho, \bar m) \in \mathcal C(\rho^0, \rho^1)$ such that $\bar \rho_1 \equiv 0$ and $\mathcal A(\bar \rho, \bar m)<\mathcal A(\rho, m).$  Let then  assume $(\rho, m) \in \mathcal C(\rho^0, \rho^1)$ is such that $\rho_1 \not \equiv 0.$ We have  
\begin{equation}\label{eq:dec07.2017.0}
\dot  \rho_1+   m_{21}=0, \quad \dot \rho_2+ m_{12} + m_{32}=0, \quad \dot \rho_3+  m_{23} =0.
\end{equation} 
Set 
\[
(\bar \rho_1, \bar \rho_2, \bar \rho_3):= (0, \rho_1+\rho_2, \rho_3),  \quad (\bar m_{12}, \bar m_{23}):=(0, m_{23}),\quad  (\bar m_{21}, \bar m_{32}):=(0, m_{32}).
\] 
Note 
$$\bar \rho(0)=\rho^0, \quad \bar \rho(1)=\rho^1 \quad \text{and}\quad \dot {\bar \rho}_1+ \bar m_{21}=0.$$ 
We use \eqref{eq:dec07.2017.0} to infer  
\[
\dot {\bar \rho}_2+ \bar m_{12} +\bar m_{32}= \dot \rho_1+ \dot \rho_2+ \bar m_{32}= -m_{21}-m_{12}-m_{32}+ m_{32}=0.
\]
Similarly, 
\[
\dot {\bar \rho}_3+ \bar m_{32}= \dot \rho_3+ m_{32}=0.
\]
Thus, we  verified that $(\bar \rho, \bar m) \in \mathcal C(\rho^0, \rho^1)$ and $\bar \rho_1 \equiv 0.$ Note we cannot have $m_{12} \equiv 0$ otherwise, we would have $\dot \rho_1 \equiv 0$ which would imply $\rho_1(t)=\rho(0)=0.$ Therefore
\begin{equation}\label{eq:dec07.2017.1}
2\mathcal A( \rho, m)=\int_0^1\Bigl({m_{12}^2 \over g(\rho_1, \rho_2)} +{m_{23}^2 \over g(\rho_2, \rho_3)}\Bigr) dt> \int_0^1 {m_{23}^2 \over g(\rho_2, \rho_3)} dt= \int_0^1 2 {m_{23}^2 \over  \rho_2+\rho_3} dt.
\end{equation} 
We have 
\begin{equation}\label{eq:dec07.2017.2}
2\mathcal A( \bar \rho, \bar m)= \int_0^1 {\bar m_{23}^2 \over g(\bar \rho_2, \bar \rho_3)} dt= \int_0^1 2 {m_{23}^2 \over  \rho_1+ \rho_2+\rho_3} dt \leq \int_0^1 2 {m_{23}^2 \over   \rho_2+\rho_3} dt.
\end{equation} 
By \eqref{eq:dec07.2017.1} and \eqref{eq:dec07.2017.2}, $\mathcal A( \rho, m)>\mathcal A( \bar \rho, \bar m).$ We conclude the proof of the proposition thanks to the  observation that since $\rho^0$ and $\rho^1$ do not have the same $g$--connected components, the $g$--connected components of $\rho^*(t)$ cannot be constant in $t$. \endproof

\begin{remark}\label{re:example-bdry2} Let  $G=(V, E, \omega)$ denote the weighted graph  in Proposition  \ref{pr:example-bdry} and let $g$ denote the function used there.  Suppose $\rho^0, \rho^1 \in \mathcal P(G),$  $(\rho, m)$ minimizes $\mathcal A$ over $\mathcal C(\rho^0, \rho^1),$ and the range of $\rho$ is entirely contained in the interior of $\mathcal P(G)$. Then using \eqref{eq:dec07.2017.0} we have 
\[
\mathcal A( \rho, m)=\int_0^1 \Bigl({(\dot \rho_1)^2 \over \rho_1+ \rho_2} +{(\dot \rho_3)^2 \over \rho_3+ \rho_2}\Bigr) dt=\int_0^1 \mathcal L_0\bigl((\rho_1, \rho_3), (\dot \rho_1, \dot \rho_3)\bigr) dt=: \mathcal A_0(\rho_1, \rho_3)
\]
where
\[
\mathcal L_0(q, u):={u_1^2 \over 1- q_3}+ {u_3^2 \over 1- q_1}, \quad q=(q_1, q_3),\,\,u=(u_1,u_3).
\]
From $\rho_1$ and $\rho_3$ we recover $\rho_2=1-(\rho_1+\rho_3).$  We have that $(\rho_1, \rho_3)$ minimizes $\mathcal A_0$ over the set of $(\rho_1, \rho_3): [0,1] \rightarrow  Q$ where $Q$ is given by \eqref{eq:dec08.2017.1}.
\end{remark}

\begin{proposition}\label{pr:necessary-suff} Let $\mathcal H_g$ be as in \eqref{eq:Hamiltonian-g} and $H$ be as in \eqref{eq:defn-of-H}. Let $\rho^0, \rho^1\in\mathcal{P}(G)$ be such that $(\rho, m) \in \mathcal C(\rho^0, \rho^1).$ Assume $\lambda \in H^1(0,1; \mathbb R^n)$ is such that $H(\dot \lambda, \nabla_G \lambda) \leq 0$ almost everywhere.
\begin{enumerate}
\item[(i)]   We have  
\[ 
(\lambda(1), \rho^1) -(\lambda(0), \rho^0) \leq \mathcal A(\rho, m).
\] 
\item[(ii)] Equality holds in (i) if and only if 
\[ 
m_{ij}= g(\rho_i, \rho_j)(\nabla_G \lambda)_{ij} \quad \forall (i, j) \in E, \quad H(\dot \lambda, \nabla_G \lambda)=(\rho, \dot \lambda)+ {1\over 2}\| \nabla_G \lambda\|_\rho^2=0 \quad \text{a.e.}.
\] 
\item[(iii)] If  the range of $\rho$ is almost everywhere contained in $(0,\infty)^n$ and $(\rho, \lambda)$ satisfies almost everywhere the Hamiltonian system 
\begin{equation}\label{eq:dec06.2017.4new}
\left\{
\begin{array}{ll}
\dot \rho&=\nabla_\phi \mathcal H_g(\rho, \lambda) \\ 
\dot \lambda &=-\nabla_\rho \mathcal H_g(\rho, \lambda),
\end{array}
\right.
\end{equation}
then equality holds in (i) and so, $(\rho, m)$ minimizes $\mathcal A$ over $\mathcal C(\rho^0, \rho^1)$ where $m_{ij}= g(\rho_i, \rho_j)(\nabla_G \lambda)_{ij} $ for any $(i, j) \in E.$ 
\end{enumerate} 
\end{proposition}
\proof{} (i) We have 
\[
(\lambda(1), \rho^1) -(\lambda(0), \rho^0) = \int_0^1 \bigl((\dot \rho, \lambda)+(\rho, \dot \lambda) \bigr) dt= \int_0^1 \Bigl(-(\textrm{div}_G(m), \lambda)+(\rho, \dot \lambda) \Bigr) dt.
\] 
Integrating by parts and then using Remark \ref{re:Legendre-trans} (ii) in the subsequent identity, we conclude 
\begin{eqnarray} 
(\lambda(1), \rho^1) -(\lambda(0), \rho^0)
&=&\int_0^1 \Bigl((m, \nabla_G \lambda) +(\rho, \dot \lambda) \Bigr) dt\nonumber\\
&\leq&\int_0^1 \Bigl( {1\over 2} F(\rho, m) +{1\over 2} \|\nabla_G \lambda\|^2_\rho+(\rho, \dot \lambda) \Bigr) dt  \label{eq:dec06.2017.6newa}\\
&\leq&\int_0^1 \Bigl( {1\over 2} F(\rho, m)+ H(\dot \lambda, \nabla_G \lambda) \Bigr) dt \leq  \mathcal A(\rho, m) .\label{eq:dec06.2017.6newb}
\end{eqnarray}
This, verifies (i). 

(ii) Note that equality holds in (i) if and only if equality hold in \eqref{eq:dec06.2017.6newa} and \eqref{eq:dec06.2017.6newb}. Using Remark \ref{re:Legendre-trans} (ii), we conclude the proof of (ii). 

(iii) Assume $(\rho, \lambda)$ satisfies almost everywhere the Hamiltonian system \eqref{eq:dec06.2017.4new}.  We use and then use Lemma \ref{le:Legendre-trans} 
\[
0=\bigl( \rho, \dot \lambda+ \nabla_\rho \mathcal H_g(\rho, \lambda) \bigr)=( \rho, \dot \lambda)+ \bigr( \rho, \nabla_\rho \mathcal H_g(\rho, \lambda) \bigr)= ( \rho, \dot \lambda) +{1\over 2}\|\nabla_G \lambda\|_\rho^2 \leq H(\dot \lambda, \nabla_G \lambda) \leq 0.
\]
Thus, 
\[
0= ( \rho, \dot \lambda) +{1\over 2}\|\nabla_G \lambda\|_\rho^2 \leq H(\dot \lambda, \nabla_G \lambda).
\]
Setting $m_{ij}= g(\rho_i, \rho_j)(\nabla_G \lambda)_{ij} $ for any $(i, j) \in E$ we use (ii) to conclude the proof of (iii).\endproof

\begin{proposition}\label{pr:nonempty} Let  $G=(V, E, \omega)$ denote the weighted graph  in Proposition  \ref{pr:example-bdry} and let $g$ denote the function used there. Let $\mathcal H_g$ be as in \eqref{eq:Hamiltonian-g} and $H$ be as in \eqref{eq:defn-of-H}. There exist  $\rho^0, \rho^1$  in interior of $\mathcal P(G)$ and there is a geodesic of minimal norm $\rho:[0,1] \rightarrow  \mathcal P(G)$, connecting $\rho^0$ to $\rho^1$ which intersects the boundary of $\mathcal P(G)$.      
\end{proposition}
\proof{} The comments in Remark   \ref{re:example-bdry2} led us to the following considerations which rely on the Lagrangian $\mathcal L_0$ introduced there.

Set $q:=(q_1, q_3)$ and choose $\delta \in (0,\; 0.1]$ such that the system of differential equations 
\begin{equation}\label{eq:dec07.2017.4}
{d \over dt} \nabla_u \mathcal L_0(q, \dot q)= \nabla_q \mathcal L_0(q, \dot q) \quad \text{on} \quad  (-\delta, \delta ), 
\end{equation} 
together with the initial conditions   
\begin{equation}\label{eq:dec07.2017.4late} 
q(0)=(0,\; 0.5),\quad \dot q(0)= (0, 1)
\end{equation}  
has a unique solution. We have the conserved quantity  
\[
{\dot q_1^2 \over 1-q_3}+{\dot q_3^2 \over 1-q_1}= 1.
\] 
For $\delta$ small enough, we have 
\begin{equation}\label{eq:dec07.2017.4c}
|q_1|\leq 0.085, \qquad  0.48 \leq q_3 \leq 0.62 \qquad \text{on} \quad [-\delta, \delta].
\end{equation} 
By \eqref{eq:dec07.2017.4}  
\begin{equation}\label{eq:dec07.2017.5}
\ddot q_1=-{\dot q_1 \dot q_3 \over 1-q_3}+{1\over 2} {\dot q_3^2 (1-q_3) \over (1-q_1)^2} \quad \text{and} \quad 
\ddot q_3=-{\dot q_1 \dot q_3 \over 1-q_1}+ {1\over 2} {\dot q_1^2 (1-q_1) \over (1-q_3)^2}.
\end{equation} 
We use \eqref{eq:dec07.2017.4c} in \eqref{eq:dec07.2017.5} to obtain a constant $C_1>0$ independent of $\delta \in (0, \; 0.1)$ such that 
\begin{equation}\label{eq:dec07.2017.5d}
|\dot q_1| + |\ddot q_1|+ |\dot q_3| + |\ddot q_3| \leq C_1 \qquad \text{on} \quad [-\delta, \delta].
\end{equation} 
Differentiating the expressions in  \eqref{eq:dec07.2017.5}, we obtain explicit expressions of $d^3 q_1/ dt^3$ and  $d^3 q_1/ dt^3$  in terms of $q_1, q_3, \dot q_1, \dot q_3, \ddot q_1, \ddot q_3.$ We use the identities in \eqref{eq:dec07.2017.5d} in these expressions to obtain a constant $C>0$ such that 
\begin{equation}\label{eq:dec07.2017.5e}
|\dot q_1| + |\ddot q_1|+\Big|{d^3 q_1 \over dt^3} \Bigr|+ |\dot q_3| + |\ddot q_3| +\Big|{d^3 q_3 \over dt^3} \Big|\leq C \qquad \text{on} \quad [-\delta, \delta].
\end{equation} 
By \eqref{eq:dec07.2017.4late}  and \eqref{eq:dec07.2017.5}, $\ddot q_1(0)=0.25$. This, together with \eqref{eq:dec07.2017.5e} implies 
\[
\ddot q_1  \geq 0.25-C t \qquad \forall t \in [-\delta, \delta].
\] 
Thus, choosing  $\delta_1$ strictly between $0$ and $\min\{ \delta, \; 0.15 C^{-1}\}$ we have 
\[
\ddot q_1  \geq 0.1 \qquad \text{on} \quad [-\delta_1, \delta_1].
\]
Since $\dot q_1(0)=0$ then for any $t \in  [-\delta_1, \delta_1]$,  
\[
q_1(t)\geq  {0.1 t^2 \over 2} \qquad \text{on} \quad [-\delta_1, \delta_1].
\]
This, together with \eqref{eq:dec07.2017.4c} yields 
\begin{equation}\label{eq:dec07.2017.5g}
0.05t^2   \leq q_1(t) \leq 0.085, \;\; 0.48\leq  q_3(t)  \leq 0.62 \;\; \text{and} \;\; q_1(t) +q_3(t)  \leq 0.705,\; \;\forall t \in [-\delta_1, \delta_1].
\end{equation} 
Setting 
$$q_2:=1-q_1-q_3 \quad \text{and} \quad \tilde q:=(q_1, q_2, q_3).$$   
\eqref{eq:dec07.2017.5g} implies 
\begin{equation}\label{eq:dec11.2017.-1} 
\tilde q \in C^3( [-\delta_1, \delta_1], \mathbb R^n),  \qquad \tilde q \bigl( [-\delta_1, \delta_1] \setminus \{0\}\bigr) \subset \mathcal P_0(G), \qquad \tilde q(0) \in \mathcal P(G) \setminus  \mathcal P_0(G).
\end{equation}
Define 
\begin{equation}\label{eq:dec11.2017.0}
l_1(t)= -\int_0^t {\dot q_1^2 \over (1-q_3)^2} ds, \quad l_2(t)= l_1(t)- {2 \dot q_1 \over 1-q_3}, \quad l_3(t):=l_2(t) + {2\dot q_3 \over 1-q_1}.
\end{equation} 
Observe 
\begin{equation}\label{eq:dec11.2017.1}
l_2-l_1=  {-2 \dot q_1  \over 1-q_3}, \quad l_3-l_2 ={2 \dot q_3 \over 1-q_1}.
\end{equation} 
Differentiating the expressions in \eqref{eq:dec11.2017.0} and using \eqref{eq:dec11.2017.1} we obtain 
\[
\dot l_1=-{1\over 4}(l_1-l_2)^2, \quad  \dot l_2=  -{1\over 4}(l_1-l_2)^2-{1\over 4}(l_3-l_2)^2 , \quad 
\dot l_3=-{1\over 4}(l_3-l_2)^2.
\] 
Thus, 
\begin{equation}\label{eq:dec11.2017.2}
\dot l=-\nabla_\rho \mathcal H_g(\tilde q, l).
\end{equation} 

We have 
\[
\partial_{\phi_1} \mathcal H_g(\tilde q, l)=g(q_1, q_2)(l_1-l_2)={q_1 +q_2 \over 2}(l_1-l_2)={1-q_3 \over 2}(l_1-l_2).
\] 
This, combined with the first identity in \eqref{eq:dec11.2017.1} yields, 
\[
\partial_{\phi_1} \mathcal H_g(\tilde q, l)= \dot q_1.
\]  
Analogously,  computing $\partial_{\phi_2} \mathcal H_g$ and $\partial_{\phi_3} \mathcal H_g$ and using  \eqref{eq:dec11.2017.1} we obtain 
\begin{equation}\label{eq:dec11.2017.3}
\dot {\tilde q}=\nabla_\phi \mathcal H_g(\tilde q, l).
\end{equation} 

Set 
\[
\rho(s):= \tilde q\Bigl(2 \delta_1 s-\delta_1 \Bigr), \qquad \lambda(s):=2 \delta_1 l\Bigl(2 \delta_1 s-\delta_1 \Bigr), \qquad m_{ij}:= g(\rho_i, \rho_j) (\nabla_G \lambda)_{ij} \quad \forall (i, j) \in E.
\]
By  \ref{eq:dec11.2017.-1} 
\begin{equation}\label{eq:dec11.2017.-1new} 
\rho \in C^3\bigl( [0, 1], \mathbb R^n\bigr),  \quad q \biggl( [0, 1] \setminus \Bigl\{{1 \over 2}\Bigr\}\biggr) \subset \mathcal P_0(G), \quad 
q\Bigl({1 \over 2}\Bigr) \in \mathcal P(G) \setminus  \mathcal P_0(G).
\end{equation}
By \eqref{eq:dec11.2017.3} and \eqref{eq:dec11.2017.2} we have 
\begin{equation}\label{eq:dec11.2017.3new}
\dot \rho=\nabla_\phi \mathcal H_g(\rho, \lambda), \qquad \dot \lambda=-\nabla_\rho \mathcal H_g(\rho, \lambda). 
\end{equation} 
The latter identity implies 
\begin{equation}\label{eq:dec11.2017.2b}
H(\dot \lambda, \nabla_G \lambda)=0.
\end{equation} 
Combining \eqref{eq:dec11.2017.3new} and \eqref{eq:dec11.2017.2b}, using Proposition \ref{pr:necessary-suff}, we obtain that $(\rho, m)$ minimizes $\mathcal A$ over $\mathcal C(\rho(0), \rho(1)).$ We learn from \eqref{eq:dec11.2017.-1new} that the end points of $\rho$ are in the interior of $\mathcal P(G)$ while the range of $\rho$ intersects the boundary of $\mathcal P(G)$. \endproof

%
%
\section{Minimizer}\label{section4}
In this section, we use the same notation as in Section \ref{section2} and assume  \eqref{eq:finite energy} hold.

%
%
\begin{lemma}\label{le:upper-bound-on-m} For any  $E_0 \subset E$ and $(\rho, m) \in \mathbb R_+^n \times S^{n \times n}$ 
\[
\epsilon_0(g) \sum_{k=1}^n \rho_k \sum_{(i,j) \in E_0} f\bigl(g_{ij}(\rho),m_{ij}\bigr)    \geq \sum_{(i,j) \in E_0} m_{ij}^2.
\]
\end{lemma}
\proof{} Let $(\rho, m) \in \mathbb R_+^n \times S^{n \times n}$. Observe that to prove the lemma we only need to take into account  $(i, j) \in E_0$ such that $m_{ij} \not =0.$ In that case, we may only account for  $(i, j) \in E_0$ such that $g_{ij}(\rho)>0$. We then have
\[
g_{ij}(\rho)=g(\rho_i, \rho_j) \leq \epsilon_0(g) (\rho_i+ \rho_j) \leq \epsilon_0(g)   \sum_{k=1}^n \rho_k.
\]
The desired inequality follows since   $\epsilon_0(g)   f\bigl(g_{ij}(\rho),m_{ij} \bigr)  \sum_{k=1}^n \rho_k \geq m_{ij}^2.$\endproof
%
%
\begin{remark}\label{re:extension} Let $h \in \kerGp,$  and let $\rho^0, \rho^1 \in \mathcal P(G).$ 
\begin{enumerate}
\item[(i)] If $\phi, \tilde\phi \in \mathbb R^n$ are such that $\nabla_G \phi= \nabla_G \tilde \phi$, since $G$ is connected we obtain $a:=\phi_i-\tilde \phi_i$ is independent of $i \in V$ and so,  
$$(\phi-\tilde \phi, h)= a \sum_{i=1}^n h_i=0.$$ 
Hence, on $R(\nabla_G)$,  the linear operator defined by $L(\nabla_G\phi):= ( h, \phi) $ is well defined. Since $R(\nabla_G)$ is of finite dimension, $L$ is continuous and so, it admits a unique linear extension $L_\rho: T_\rho \mathcal P(G) \rightarrow \mathbb R,$ which is in turn continuous. 
\item[(ii)] By the Riesz representation there, there exists a unique $l_\rho(h) \in T_\rho \mathcal P(G)$ such that 
$$
L_\rho=(\cdot, l_\rho(h))_\rho.
$$ 
By the fact that $L_\rho(\nabla \phi)=(\phi, h)$ for every $\phi \in \mathbb R^n$, we have $h+ \divrho\bigl(l_\rho(h)\bigr)=0.$
\end{enumerate}
\end{remark}


Set 
\[
\mathcal E[\rho, v](h):={1\over 2} \|v \|^2_\rho -\bigl( v, l_\rho(h)\bigr)_\rho \qquad \forall v \in S^{n\times n}.
\]

\begin{proposition}\label{pr:hodge0} Let $m \in S^{n \times n}$, let $\rho \in \mathcal P(G)$ and let $h \in \kerGp.$   
\begin{enumerate}
\item[(i)] $\nabla_G \cdot (m) \in \kerGp.$ 
\item[(ii)] $l_\rho(h)$ is the unique minimizer of $\mathcal E[\rho, \cdot](h)$ over $T_\rho \mathcal P(G)$ and over $\mathbb H_\rho$.
\item[(iii)] If $w \in T_\rho \mathcal P(G)$ and $h=-\divrho (w)$ then $w=l_\rho(h).$ 
\item[(iv)] If $h=-\nabla_G \cdot (m)$ then $F(\rho, m) \geq  \|l_\rho(h)\|^2_\rho.$ 
\end{enumerate}
\end{proposition}
\proof{} (i) We use the fact that $m_{ij}+m_{ji}=0$ for any $(i,j) \in E$ and that $(i,j) \in E$ if $(j,i) \in E$ to obtain (i).

(ii) Let $v\in S^{n\times n}$ be such that $v\not = \pi_\rho(v).$ We have $\|v\|_\rho > \|\pi_\rho(v)\|_\rho$ and by the characterization of $\pi_\rho(v)$ in \eqref{eq:oct22.2017.2}, we have  $\bigl(v, l_\rho(h)\bigr)_\rho=\bigl(\pi_\rho(v), l_\rho(h)\bigr)_\rho$. Hence, 
\[
\mathcal E[\rho, v](h)> {1 \over 2}  \bigl\| \pi_\rho(v)\bigr\|^2_\rho -\bigl(\pi_\rho(v), l_\rho(h)\bigr)_\rho=  {1 \over 2}  \bigl\| \pi_\rho(v)-l_\rho(h)\bigr\|^2_\rho -{1 \over 2} \|l_\rho(h)\|^2_\rho.
\]
If $\pi_\rho(v)\not =l_\rho(h)$, we conclude that 
\[
\mathcal E[\rho, v](h)> -{1 \over 2} \|l_\rho(h)\|^2_\rho= \mathcal E[\rho, l_\rho(h)](h).
\]
This concludes the proof of (ii). 

(iii) Since 
\[
\bigl(\nabla_G \phi, l_\rho(h)\bigr)_\rho=(\phi, h )=\bigl(\nabla_G \phi,w\bigr)_\rho \qquad \forall \phi \in \mathbb R^n, 
\]
we have  
\[
\bigl(v, l_\rho(h)-w\bigr)_\rho=0 \qquad \forall v\in T_\rho \mathcal P(G), 
\]
which proves (iii). 

(iv)  Assume $h=-\nabla_G \cdot (m)$ and set $ l_\rho(h)=w.$ If $F(\rho, m)=\infty$ there is nothing to prove. Assume that $F(\rho, m)<\infty$ so that $m_{ij}=0$ whenever $g_{ij}(\rho)=0$. There is a unique vector field $v$ such that $g_{ij}(\rho) v_{ij}= m_{ij}$ and $v_{ij}=0$ whenever $g_{ij}(\rho)=0$. We have 
\begin{equation}\label{eq:oct22.2017.4}
F(\rho, m)=\sum_{g_{ij}(\rho)>0} {m^2_{ij} \over 2 g_{ij}(\rho)}={1\over 2} \sum_{g_{ij}(\rho)>0} g_{ij}(\rho) v^2_{ij}={1\over 2} \sum_{(i,j) \in E} g_{ij}(\rho) v^2_{ij}= \|v\|^2_\rho \geq 
\bigl\|\pi_\rho(v)\bigr\|^2_\rho.
\end{equation}
Since 
\[h=-\nabla_G \cdot (m)=-\divrho (v)=  -\divrho \bigl(\pi_\rho(v)\bigr), \]
(iii) implies $\pi_\rho(v)=w$. This, together with \eqref{eq:oct22.2017.4} proves (iv). \endproof

\begin{remark}\label{rem:conA} The following remarks are needed in the manuscript.
\begin{enumerate}
\item[(i)] $\rho \in L^2(0,1; \mathbb R^n)$ and $m\in L^2(0,1; S^{n \times n})$ are such that $\mathcal A(\rho, m)<\infty$ then $ f\bigl(g_{ij}(\rho) ,m_{ij}\bigr) \in    L^1(0,1)$ and for any $(i,j)\in E$, 
\begin{equation}\label{eq:claim1}
\mathcal{L}^1\Bigl(\Bigl\{t\in(0,1)\;\; \Big| \;\; g_{ij}( \rho(t))=0, \;\; m_{ij}(t)\not=0 \Bigr\}\Bigr)=0. 
\end{equation}
\item[(ii)] $\mathcal A$ is non--negative and lower semicontinuous on $L^2(0,1; \mathbb R^n) \times L^2(0,1; S^{n \times n})$ for the weak convergence.
\end{enumerate}
\end{remark}
\proof{} We skip the proof of (i). Since  $F$ is a nonnegative convex, lower semicontinuous function, by standard theory of the calculus of variations (cf. e.g. \cite{EkelandT}) we obtain (ii).\endproof

\begin{theorem}\label{th:exist} Assume  $\rho^0, \rho^1 \in \mathcal P(G).$ 
\begin{enumerate}
\item[(i)] There exists $(v^*, \rho^*, m^*)$ such that   $(v^*, \rho^*)$ is a minimizer in \eqref{2-W dist} and $(\rho^*, m^*)$ minimizes $\mathcal A$ over $\mathcal C(\rho^0, \rho^1).$  
\item[(ii)] Furthermore,  
\[
2\inf_{\mathcal C(\rho^0, \rho^1)} \mathcal A= \mathcal W_g^2(\rho^0, \rho^1)= \int_0^1 \bigl\| v^*\bigr\|^2_{\rho^*} dt=2\mathcal A(\rho^*,m^*).
\]
\item[(iii)] We have  
\begin{equation*}
F(\rho^*, m^*)(t)= F(\rho^*, m^*)(0) \quad \text{a.e. on} \quad (0,1). 
\end{equation*}
\end{enumerate}
\end{theorem}
\proof{} By Proposition \ref{pro:finite-energy-n-vertices}, there is a positive number $i_0$ and a path $(\rho, m) \in \mathcal C(\rho^0, \rho^1)$  such that $\mathcal A(\rho, m) \leq i_0$. By Lemma \ref{le:upper-bound-on-m},  
\[
\int_0^1 |m(t)|^2 dt \leq  2\epsilon_0(g) i_0.
\]
Using the differential equation linking $ \dot \rho$ to $m$ we conclude that  for a constant $C$ depending only on $i_0$,  $\epsilon_0(g)$, $w$ and $n$, we have 
\[
  \bigl\|\dot \rho  \bigr\|^2_{L^2(0,1)} \leq C.
\]
Increasing the value of $C$ is necessary, we use the Poincar\'e--Wintiger inequality to obtain 
\[
\|\rho \|^2_{H^1(0,1)} \leq C.
\]
As a consequence, the intersection of $\mathcal C(\rho^0, \rho^1)$ with any sub--level subsection of $\mathcal A$ is precompact set of $H^1(0,1; \mathbb R^n) \times L^2(0,1; S^{n \times n})$ for the weak topology. By Remark \ref{rem:conA} (ii), $\mathcal A$ is weakly lower semicontinuous and so, it achieves its minimum at some $(\rho^*, m^*) \in \mathcal C(\rho^0, \rho^1).$

Since $\int_0^1 F(\rho^*(t), m^*(t))<\infty$, by Remark \ref{rem:conA} (i), the set obtained as the union over $(i,j)\in E$ of the sets $\{g_{ij}( \rho^*)=0\} \cap \{ m_{ij}^*\not=0\}$ is of null measure. Thus the functions $v^*_{ij}:(0,1) \rightarrow \mathbb R$ defined as  
\begin{equation}\label{eq:minimizer1}
v^*_{ij}(t)= 
\left\{
\begin{array}{ll}
{m^*_{ij}(t) \over g_{ij}(\rho^*(t)) }  & \quad\mbox{if}\,\, g_{ij}(\rho^*(t))>0\\ 
  &      \\ 
0  &\quad\mbox{if}\,\, g_{ij}(\rho^*(t))=0
\end{array}
\right.
\end{equation}
 are measurable and satisfies $m^*_{ij}= g_{ij}(\rho^*) v^*_{ij}.$ Let $(\rho, v)$ be an admissible path in \eqref{2-W dist}. This means we are assuming that $(\rho(0), \rho(1))=(\rho^0, \rho^1),$ $\rho \in H^1(0,1; \mathbb R^n)$, $v:[0,1] \rightarrow S^{n \times n}$ is Borel measurable  and 
\[
\int_0^1 \|v\|^2_ \rho dt<\infty.
\]
 Setting $m_{ij}=g_{ij}(\rho) v_{ij}$ we have 
 \[
 \int_0^1 |m|^2 dt =   {1\over 2}\sum_{(i,j) \in E} \int_0^1 v_{ij}^2 g^2_{ij}(\rho) dt \leq \epsilon_0(g) \sum_{(i,j) \in E} \int_0^1 v_{ij}^2 g_{ij}(\rho) dt=2 \epsilon_0(g)  \int_0^1 \|v\|^2_ \rho dt.
 \] 
 Thus, $m \in L^2(0,1;S^{n \times n})$ and so,  $(\rho, m) \in \mathcal C(\rho^0, \rho^1).$ By the definition of $v^*$
 \begin{equation}\label{eq:minimizer1new2}
  \|v^*\|^2_ {\rho^*} =   F(\rho^*, m^*).
\end{equation}
 By the minimality property of $(\rho^*, m^*)$, we have  
\begin{equation}\label{eq:minimizer1new}
  \int_0^1 \|v^*\|^2_ {\rho^*} dt = 2   \mathcal A(\rho^*, m^*) \leq  2   \mathcal A(\rho, m)=  \int_0^1 \|v\|^2_ {\rho} dt.
\end{equation}
 This proves (i) and also (ii). 
 
(iii) Here, we borrow ideas from \cite{chenGGT}. Let $\zeta \in C_c^1(0,1)$ be arbitrary and set $S(t)=t+\epsilon \zeta(t).$ We have $S(0)=0,$ $S(1)=1$ and $\dot S(t)=1+\epsilon \dot \zeta(t)>1/2$ for $|\epsilon|<<1.$ Thus, $S:[0,1] \rightarrow [0,1]$ is a diffeomorphism. Let $T:=S^{-1}$ and set
\[
f(s)= \rho^*(T(s)), \quad w(s)= \dot T(s) m^*(T(s)).
\] 
We have 
\[
\dot f + \nabla_G \cdot (w)=0, \quad f(0)=\rho^0, \; f(1)=\rho^1.
\] 
Thus, $(f, w) \in \mathcal C(\rho^0, \rho^1)$ and so, 
\[
\int_0^1 F( \rho^*, m^*) dt \leq \int_0^1 F ( f, w) ds=  \int_0^1 \dot T^2 F( \rho^*(T), m^*(T)) ds.
\]
We use the fact that $dt= \dot T(s) ds$ and $\dot T(S(t)) \dot S(t)=1$ to conclude that
\[
\int_0^1 F( \rho^*, m^*) dt  \leq \int_0^1 {1 \over \dot S} F ( \rho^*, m^*) dt= \int_0^1 (1- \epsilon \dot \zeta + o(\epsilon)) F( \rho^*, m^*) dt.
\]
Since $\epsilon \rightarrow \int_0^1 (1- \epsilon \dot \zeta + o(\epsilon)) F ( \rho^*, m^*)  dt$ admits its minimum at $0$, we conclude that its derivative there is null, i.e.,
\[
\int_0^1 \dot \zeta(t) F ( \rho^*,  m^*) dt=0.
\]
This proves that the distributional derivative of $F ( \rho^*,  m^*)$ is null and so, $F ( \rho^*,  m^*)$ is independent of $t$. \endproof

\begin{remark}\label{re:exist} Let $(\rho^*, m^*, v^*)$ as in Theorem \ref{th:exist}
\begin{enumerate}
\item[(i)] We have $\|v^*\|^2_{\rho^*}= F(\rho^*, m^*)=\mathcal W_g^2(\rho^0, \rho^1)$ and so,  $\|v^*\|_{\rho^*} \in L^\infty(0,1)$. 
\item[(ii)] We have $ \rho^* \in W^{1,\infty}(0,1; \mathbb R^n),$ $m^* \in L^\infty(0,1; S^{n\times n})$  and  
\[
\bigl|m^*_{ij}\bigr| \leq \mathcal W_g(\rho^0, \rho^1)\; \sqrt{ \max_{[0,1]^2} g}.
\]
\end{enumerate}
\end{remark}
\proof{} (i) is a direct consequence of Theorem \ref{th:exist} (iii) and \eqref{eq:minimizer1new2}.

(ii)  The same argument which led to the definition of $v^*$ in \eqref{eq:minimizer1} can be used to conclude that if $(i,j) \in E$ then either $m^*_{ij}(t)=0$ or 
\[
{\bigl(m^*_{ij}\bigr)^2\over  g(\rho^*_i, \rho^*_j)} (t) \leq F(\rho^*, m^*)(t)=\mathcal W_g^2(\rho^0, \rho^1).
\] 
In any of these two cases,  
\[
\bigl(m^*_{ij}\bigr)^2 \leq \mathcal W_g^2(\rho^0, \rho^1) g(\rho^*_i, \rho^*_j)   \leq \mathcal W_g^2(\rho^0, \rho^1) \max_{[0,1]^2} g, 
\]
which yields the desired inequality in (ii). This shows $m^* \in L^\infty(0,1; S^{n\times n}),$  which together with the identity $\dot \rho=-\nabla_G \cdot m^*,$ shows $\dot \rho \in L^\infty(0,1; \mathbb R^n).$\endproof

%
%
\section{Duality in a Smooth Setting}\label{section5} 
Throughout this section we further assume that $g$ satisfies \eqref{eq:finite energy}. The main purpose of the section is to find the Euler--Lagrange equations satisfied by geodesics of minimal action connecting $\rho^0, \rho^1 \in \mathcal P(G)$. We will express these Euler--Lagrange equations in terms of the Hamiltonian $H$ defined in \eqref{eq:defn-of-H}. It is convenient to set 
\begin{equation*}
A:=\Bigl\{\rho\in L^2\left(0,1;\mathbb R^{n}\right)\; \big| \; \sum_{i=1}^{n}\rho_i=1,\;\; \rho_i \geq 0\; \forall i=1, \cdots, n\Bigr\}  \times L^2\left(0,1; S^{n\times n}\right), \quad B:=H^1\left(0,1;\mathbb R^{n}\right).
\end{equation*} 
For $l\in (0, \infty)$ we set 
\begin{equation}\label{eq:defnA-l}
A^l:=\biggl\{\rho\in L^2\left(0,1;\mathbb R^{n}\right)\;\; \Big| \,\,\sum_{i=1}^{n}\rho_i\leq l,\;\; \rho_i \geq 0\; \forall i=1, \cdots, n\biggr\}\times L^2\left(0,1; S^{n\times n}\right),
\end{equation}
and
\[
B_l:=\bigl\{\lambda\in B\;\; | \;\;\|\lambda\|_{H^1\left(0,1;\mathbb R^{n}\right)}\leq l\bigr\}.
\]

We plan to prove the duality property  
\begin{equation}\label{eq:min m}
\min_{\mathcal C(\rho^0, \rho^1)} \mathcal A= \sup_{\lambda\in B}\left\{\left(\lambda(1),\rho^1\right)-\left(\lambda(0),\rho^0\right)\;\; \Big| \;\;H\left(\dot \lambda,\nabla_G\lambda\right)=0\right\}.
\end{equation}
As we will show, this reduces to a minimax identity for 
\begin{equation*}
\mathcal{L}(\rho,m,\lambda):=\left(\lambda(1), \rho^1\right)-\left(\lambda(0), \rho^0\right)+ \mathcal A(\rho,m)- 
\int_0^1\left( (\dot \lambda ,\rho)+\left( m,\nabla_G\lambda\right)\right) dt.
\end{equation*}

\begin{proposition}\label{thm:minmax both bdd} For any $l>0$, $e\geq 1$  and $A_* \in \{A, A^{e}\}$ we have  
\[
\inf_{(\rho,m)\in A_*}\;\; \sup_{\lambda\in B_l}\mathcal{L}(\rho,m,\lambda)=\sup_{\lambda \in B_l}\;\; \inf_{(\rho,m)\in A_*}\mathcal{L}(\mathcal{\rho},m,\lambda).
\]
\end{proposition}
\begin{proof} Let $(\bar \rho, \bar m)\in A_*$, $\bar \lambda\in B_l$ and $C \in \mathbb R$ be arbitrary. To show the proposition, according to the standard minimax theorem (cf. e.g. )\cite{MertensSZ}), it suffices to show the following properties: $B_l$ is a convex set, compact set for the weak topology (which is obviously the case), $A_*$ is a convex topological space (which is obviously the case for the weak topology), $\{\lambda\in B_l\; | \;\mathcal{L}(\bar \rho, \bar m,\lambda)\geq C\}$ is a closed convex set in $B_l$ and $\{(\rho,m)\in A_*\; |  \,\mathcal{L}(\rho,m,\bar \lambda)\leq C\}$ is a closed convex set in $A_*.$  

When $\mathcal A (\bar \rho, \bar m)=\infty$ then $\mathcal{L}(\bar \rho, \bar m,\cdot) \equiv \infty$ and so, $\{\lambda\in B_l\,|\,\mathcal{L}(\bar \rho, \bar m,\lambda)\geq C\}=B_l$ is a closed convex set. When $\mathcal A(\bar \rho, \bar m)<\infty,$ $\mathcal{L}(\bar \rho, \bar m,\cdot)$ is linear and so, $\{\lambda\in B_l\,|\,\mathcal{L}(\bar \rho, \bar m,\lambda)\geq C\}$ is convex. We use the fact that bounded subsets of $H^1(0,1)$ are compact in $C[0,1]$, that $\bar m \in L^2$ and $\bar \rho \in L^\infty$ to conclude that $\mathcal{L}(\bar \rho, \bar m,\cdot)$ is a continuous function on $B_l$ and so, $\{\lambda\in B_l\,|\,\mathcal{L}(\bar \rho, \bar m,\lambda)\geq C\}$ is a closed subset of $B_l.$

By the fact that $F$ is a convex function and $A_*$ is a convex set,  $\{(\rho,m)\in A_*\,|\,\mathcal{L}(\rho,m, \bar \lambda)\leq C\}$ is convex subset of $A_*.$ One part of $\mathcal{L}(\cdot, \cdot,\bar \lambda)$ is a linear functional and by Remark \ref{rem:conA} (ii), the other part is weakly lower semicontinuous. Thus, $\mathcal{L}(\cdot, \cdot,\bar \lambda)$ is itself weakly lower semicontinuous and so, $\{(\rho,m)\in A_*\; | \; \mathcal{L}(\rho,m,\bar \lambda)\leq C\}$ is a closed subset of $A_*$. This, concludes the proof of the Proposition. 
\end{proof}

Since $\mathcal{L}$ is linear with respect to $\lambda$, we have 
\begin{equation}\label{eq:property-of-E-0}
\sup_{B_l}\mathcal{L}(\rho,m,\lambda)= \mathcal A(\rho,m)+l\mathcal{E}(\rho,m)
\end{equation} 
where 
\[
\mathcal{E}(\rho,m)= \sup_{\lambda \in B_1}\left(\lambda(1),\rho^1\right)-\left(\lambda(0),\rho^0\right)-\int_0^1\left( (\dot \lambda ,\rho )+( m,\nabla_G \lambda) \right) dt.
\]
Observe that 
\begin{equation}\label{eq:property-of-E}
\mathcal{E}(\rho,m)=\begin{cases}
\hfill 0 & \text{if} \quad (\rho, m) \in \mathcal C(\rho^0, \rho^1)\\
>0 & \text{otherwise.}
\end{cases}
\end{equation}
\begin{remark}\label{rem:minA-El} Let $l>0$, $e\geq 1$ and $A_*\in\{A,A^e\}$. By Theorem \ref{th:exist}, $\mathcal A$ achieves its minimum over $\mathcal C(\rho^0, \rho^1)$ at some $(\rho^*, m^*).$ By \eqref{eq:property-of-E}, we obtain that $\mathcal E(\rho^*, m^*)=0$ and thus the infimum of $\mathcal A +l \mathcal E$ over $A_*$ is between $0$ and $\mathcal A(\rho^*, m^*).$
\end{remark}

\begin{lemma}\label{lem:E lsc} Let $e \geq 1$ and let $A_* \in \{A, A^e\}.$ The following hold. 
\begin{enumerate}
\item[(i)] $\mathcal{E}$ is convex and weakly lower semi--continuous on $L^2(0,1;\mathbb R^{n})\times L^2(0,1; S^{n\times n})$.
\item[(ii)] For any $l>0$, there exists $(\rho^{*,l},m^{*,l})$ which minimizes $\mathcal A +l \mathcal E$ over $A_*$. 
\item[(iii)] The set $\{(\rho^{*,l},m^{*,l})\; | \; l > 0 \}$ is pre-compact in $A_*.$
\end{enumerate}
\end{lemma}
\begin{proof} (i) As a supremum of continuous linear functionals, $\mathcal{E}$ is convex and weakly lower semi--continuous.

(ii) By Remark \ref{rem:minA-El}, $\mathcal A +l \mathcal E$ is not identically $\infty$ over $A_*.$ Since $\mathcal E\geq 0$, any sub-level subset of $\mathcal A + l \mathcal E$ is a sub-level subset of $\mathcal A$. By  Lemma \ref{le:upper-bound-on-m}, the  sub-level sets of $\mathcal A$ are contained in a bounded subset of $L^2(0,1; \mathbb R^n) \times L^2(0,1; S^{n\times n})$ and so, they are pre--compact. Thus, the sub-level subsets of $\mathcal A + l \mathcal E$ are pre--compact. By Remark \ref{rem:conA} $\mathcal A$ is weakly lower semi--continuous  on $L^2(0,1;\mathbb R^{n})\times L^2(0,1; S^{n\times n})$ and by (i) $\mathcal{E}$ is weakly lower semi--continuous  on that same set. That is all we need to prove that $\mathcal A +l \mathcal E$ achieves its minimum over the closed set $A_*$. 

(iii) By Remark \ref{rem:minA-El}, for any $l>0$,  
\[
\mathcal A(\rho^{*,l},m^{*,l}) \leq \mathcal A(\rho^{*},m^{*}).
\] 
We apply again Lemma \ref{le:upper-bound-on-m} to conclude that $\{ m^{*,l}\; |\; l> 0\}$ is bounded in $L^2$. This is sufficient to verify (iii). 
\end{proof}

\begin{lemma}\label{le:minmax one bdd} Let $e \geq 1$ and let $A_* \in \{A^e, A\}.$ We have 
\[
\inf_{(\rho,m)\in A_*}\sup_{\lambda\in B}\mathcal{L}(\rho,m,\lambda)=\sup_{\lambda \in B}\inf_{(\rho,m)\in A_*}\mathcal{L}(\mathcal{\rho},m,\lambda).
\]
\end{lemma}
\begin{proof} We use \eqref{eq:property-of-E-0} and the first identity in \eqref{eq:property-of-E} to obtain  for any $l\geq 1,$ 
\[
\mathcal A (\rho^*,m^*)=\sup_{\lambda\in B_l}\mathcal{L}(\rho^*, m^*,\lambda) \geq \inf_{(\rho,m)\in A_*} \sup_{\lambda\in B_l}\mathcal{L}(\rho, m,\lambda).
\]
This, together with Proposition \ref{thm:minmax both bdd} implies  
\begin{equation}\label{eq:limit}
\mathcal A (\rho^*,m^*) \geq \sup_{\lambda\in B_l}\inf_{(\rho,m)\in A_*}\mathcal{L}(\rho,m,\lambda).
\end{equation}
This means  
\begin{equation}\label{eq:removelimit}
 \mathcal A (\rho^*,m^*)  \geq \mathcal A(\rho^{*,l},m^{*,l})+ l \mathcal{E}(\rho^{*,l},m^{*,l}) \geq \mathcal A(\rho^{*,l},m^{*,l}).
\end{equation}
By Lemma \ref{lem:E lsc} (iii), the second inequality in \eqref{eq:removelimit} yields that the set $\{(\rho^{*,l},m^{*,l})\; | \; l\geq 1 \}$ is pre-compact in $A_*$ and so, its admits a point of accumulation $(\rho^{\infty},m^{\infty})$. Since Lemma \ref{lem:E lsc} (i) ensures that $\mathcal{E}$ is weakly lower semi--continuous, we may divide the expression in \eqref{eq:removelimit} by $l$ and then let $l$ tend to $\infty$ in the subsequent inequality and use the fact that $\mathcal E$ is nonnegative, to obtain  $\mathcal{E}(\rho^{\infty},m^{\infty})=0.$ Thanks to \eqref{eq:property-of-E} we obtain that 
\begin{equation*}
(\rho^{\infty}, m^{\infty}) \in \mathcal C(\rho^0, \rho^1).
\end{equation*}  
By the minimality property of $(\rho^{*},m^{*})$ obtained in Theorem \ref{th:exist}, we have 
\begin{equation}\label{eq:<}
\mathcal A(\rho^*,m^*) \leq \mathcal A(\rho^\infty,m^\infty).
\end{equation}
We let $l$ tend to $\infty$ in \eqref{eq:removelimit} and use the lower semicontinuity property of $\mathcal A$ given in Remark \ref{rem:conA} to reverse the inequality in \eqref{eq:<}. 
In conclusion, 
\begin{equation*}
\mathcal A(\rho^*,m^*) = \mathcal A(\rho^\infty,m^\infty).
\end{equation*}
and
\begin{equation*}
\varliminf_{l\to+\infty}l\mathcal{E}(\rho^{*,l},m^{*,l})=0.
\end{equation*}
To summarize, we have proven that 
\[
\mathcal A(\rho^\infty,m^\infty) \leq \varliminf_{l\to+\infty} \mathcal A(\rho^{*,l},m^{*,l})+   l\mathcal{E}(\rho^{*,l},m^{*,l})= 
\varliminf_{l\to+\infty}  \inf_{(\rho, m) \in  A_*} \sup_{\lambda \in B_l} \mathcal L(\rho, m, \lambda).
\]
We first apply the duality identity in Proposition \ref{thm:minmax both bdd} to interchange $\inf_{A_*} \sup_{B_l}$ and $\sup_{B_l}  \inf_{A_*}$. Then we use the fact that Since $B_l \subset B$ to conclude that  
\begin{equation}\label{eq:nov05.2017.1}
\mathcal A(\rho^\infty,m^\infty) \leq \varliminf_{l\to+\infty} \sup_{\lambda \in B_l} \inf_{(\rho, m) \in A_*}   \mathcal L(\rho, m, \lambda) \leq 
\sup_{\lambda \in B} \inf_{(\rho, m) \in A_*} \mathcal L(\rho, m, \lambda).
\end{equation}
But as  $\mathcal L(\rho^\infty,m^\infty, \cdot) \equiv \mathcal A(\rho^\infty,m^\infty)$ we infer 
\[
\mathcal A(\rho^\infty,m^\infty)= \sup_{\lambda \in B} \mathcal L(\rho^\infty,m^\infty, \lambda) \geq  \inf_{(\rho, m) \in A_*}\sup_{\lambda \in B} \mathcal L(\rho,m, \lambda).
\]
This, together with \eqref{eq:nov05.2017.1} yields, 
\[
\inf_{(\rho, m) \in A_*}\sup_{\lambda \in B} \mathcal L(\rho,m, \lambda) \leq \sup_{\lambda \in B} \inf_{(\rho, m) \in A_*} \mathcal L(\rho, m, \lambda).
\]
The reverse inequality $\sup_{B}  \inf_{A_*}  \leq \inf_{A_*} \sup_{B}$ being always true, we conclude the proof of the lemma. 
\end{proof}
Recall $A^l$ is defined earlier in \eqref{eq:defnA-l}. Set  
\begin{equation*}
A^\infty:= \Bigl\{\rho\in L^2\left(0,1;\mathbb R^{n}\right)\;\; \big| \;\;\rho_i\geq 0\; \forall i=1, \cdots, n\Bigr\} \times L^2\left(0,1; S^{n\times n}\right).
\end{equation*}

\begin{lemma}\label{lem:ALAinfty} Let $\lambda\in B$.  
\begin{enumerate}
\item[(i)] For any $l>0$  
\[
\inf_{(\rho,m)\in A^l}\mathcal{L}(\rho,m,\lambda)=\left(\lambda(1), \rho^1\right)-\left(\lambda(0), \rho^0\right)-l \int_0^1 \bigl(H(\dot \lambda,\nabla_G\lambda) \bigr)_+dt.
\]
\item[(ii)] 
\begin{equation*}
\inf_{(\rho,m)\in A^\infty}\mathcal{L}(\rho,m,\lambda)=
\left\{
\begin{aligned}
&\left(\lambda(1), \rho^1\right)-\left(\lambda(0),\rho^0\right) \; & \text{if}\;\; \bigl(H(\dot \lambda,\nabla_G\lambda) \bigr)_+ \leq 0\,\,\text{a.e. in $(0,1)$}\\
&-\infty &\text{otherwise.}
\end{aligned}
\right.
\end{equation*}
\end{enumerate}
\end{lemma}
\begin{proof} Expressing the $\inf$ in terms of $-\sup$ and using  \eqref{eq:partial-Legendre} we have 
\begin{equation} \label{eq:nov06.2017.2}
\inf_{m \in L^2} \int_0^1\frac{1}{2}F(\rho,m)-\left(\dot \lambda ,\rho\right)-\left( m,\nabla_G\lambda\right) dt= 
- \int_0^1\Bigl( (\dot \lambda,\rho)+{1\over 2} \| \nabla_G \lambda \|^2_\rho\Bigr) dt.
\end{equation}
Since $g$  is $1$--homogeneous 
\[
\int_0^1\sup_{\sum_{i=1}^n\rho_i \leq l}\Bigl( (\dot \lambda,\rho)+{1\over 2} \| \nabla_G \lambda \|^2_\rho\Bigr) dt=l
\int_0^1\Bigl( H(\dot \lambda,\nabla_G \lambda)\Bigr)_+ dt.
\] 
This, together with \eqref{eq:nov06.2017.2}, proves (i). We let $l$ tend to $\infty$ to verify (ii). 
\end{proof}

\begin{lemma}\label{lem:tran}
Let $\lambda\in H^1\left(0,1;\mathbb R^{n}\right)$ and $\alpha\in H^1(0,1)$ and set $\bar {\lambda}_i=\lambda_i+\alpha$.
Then
\begin{equation*}
H(\dot {\bar \lambda},\nabla_G\bar\lambda)=H(\dot \lambda,\nabla_G\lambda)+\dot \alpha.
\end{equation*} 
\begin{proof} Observe that for any $\rho \in \mathcal P(G)$ we have 
\begin{equation} \label{eq:nov06.2017.3}
(\dot {\bar \lambda} ,\rho)= (\dot \lambda,\rho)+\dot \alpha \quad \text{and} \quad \nabla_G \bar \lambda=\nabla_G \lambda.
\end{equation}
Since 
\[
H(\dot {\bar \lambda},\nabla_G\bar\lambda)=\sup_{\rho\in \mathcal{P}(G)}\Bigl\{(\dot {\bar\lambda},\rho)+\frac{1}{2} \|\nabla_G \bar\lambda\|^2_{\rho}\Bigr\},
\]
we use \eqref{eq:nov06.2017.3}  to conclude the proof. \end{proof}
\end{lemma}
\begin{proposition}\label{prop:bar lambda} Given $\lambda\in H^1\left(0,1;\mathbb R^{n}\right)$, there is $\bar\lambda\in  H^1\left(0,1;\mathbb R^{n}\right)$ such that $ H (\dot {\bar\lambda},\nabla_G\bar\lambda )_+\equiv 0$ and 
$$\inf_{(\rho,m)\in A^1}\mathcal{L}(\rho,m,\lambda)=\inf_{(\rho,m)\in A^1}\mathcal{L}(\rho,m,\bar\lambda).$$
\end{proposition}
\begin{proof}
Let 
\begin{equation*}
O:=\{t\in (0,1)\,\,|\,\,H\bigl(\dot \lambda,\nabla_G\lambda\bigr)>0\}.
\end{equation*}
If $\mathcal{L}^1(O)=0$, we are done by letting $\bar\lambda=\lambda$. Assume that $\mathcal{L}^1(O)>0$. Set 
\begin{equation*}
\alpha(t)=-\int_0^t\chi_{O}(s)H(\dot \lambda(s),\nabla_G\lambda(s))ds.
\end{equation*}
Since (H-i) holds and $\lambda\in H^1\left(0,1;\mathbb R^{n}\right)$, we have
\[
|\dot \alpha|= |  \chi_{O}H\bigl(\dot \lambda ,\nabla_G\lambda \bigr)|\leq 
\left|\sup_{\rho\in\mathcal{P}(G)}\left\{\bigl(\dot \lambda,\rho\bigr)+\frac{1}{2}\|\nabla_G \lambda\|^2_\rho \right\}\right| 
\leq |\dot \lambda|+ \max_{[0,1]^n} \frac{g}{2}\|\nabla_G \lambda\|^2 \in L^2(0,1).
\]
Thus, $\alpha\in H^1(0,1)$. Set $\bar\lambda_i=\lambda_i+\alpha$. By Lemma \ref{lem:tran}, we obtain 
\[
H(\dot {\bar\lambda} ,\nabla_G\bar\lambda)= H (\dot \lambda,\nabla_G\lambda)+\dot \alpha=(1-\chi_O) H (\dot \lambda,\nabla_G\lambda)\leq 0.
\]
Hence, $H (\dot {\bar \lambda},\nabla_G\bar\lambda)_+\equiv 0$, which verifies the first claim of the proposition.

By Lemma \ref{lem:ALAinfty}  and (i), we infer, 
\begin{equation} \label{eq:nov06.2017.5}
\inf_{(\rho,m)\in A^1}\mathcal{L}(\rho,m,\bar\lambda)= \left(\lambda(1),\rho^1\right)-\left(\lambda(0),\rho^0\right)+\alpha(1)-\alpha(0).
\end{equation}
Since 
\begin{equation} \label{eq:nov06.2017.6}
\alpha(1)-\alpha(0)= -\int_0^1\left(H (\dot \lambda,\nabla_G\lambda)\right)_+dt,
\end{equation}
we use Lemma \ref{lem:ALAinfty}  again and combine \eqref{eq:nov06.2017.5} and \eqref{eq:nov06.2017.6} to verify (ii).
\end{proof}
\begin{proposition}\label{lem:legimply=}
Suppose $\lambda\in H^1\left(0,1;\mathbb R^{n}\right)$ and $H(\dot \lambda,\nabla_G\lambda)\leq 0.$ 
Then there exists $\bar\lambda\in H^1\left(0,1;\mathbb R^{n}\right)$ such that $H\left(\dot {\bar\lambda},\nabla_G\bar\lambda\right)= 0$
and
\begin{equation*}
\left(\bar\lambda(1),\rho^1\right)-\left(\bar\lambda(0), \rho^0\right)\geq\left(\lambda(1),\rho^1\right)-\left(\lambda(0), \rho^0\right).
\end{equation*}
\end{proposition}
\begin{proof} Let $O:=\{ H(\dot \lambda,\nabla_G\lambda)<0\}$ and to avoid trivialities, assume $\mathcal{L}^1(O)>0$. Set 
\begin{equation*}
\alpha(t)=-\int_0^t\chi_{O}H(\dot \lambda,\nabla_G\lambda)ds
\end{equation*}
and $\bar\lambda_i=\lambda_i+\alpha$. As done in the proof of Proposition \ref{prop:bar lambda}, we have $\alpha\in H^1(0,1)$ and by Lemma \ref{lem:tran}, 
\[
H(\dot {\bar \lambda},\nabla_G\bar \lambda)=(1-\chi_O) H(\dot { \lambda},\nabla_G \lambda)=0.
\]
Finally, 
\[
\left(\bar\lambda(1),\rho^1\right)-\left(\bar\lambda(0),\rho^0\right)=\left(\lambda(1),\rho^1\right)-\left(\lambda(0),\rho^0\right)-\int_{O}H(\dot { \lambda},\nabla_G \lambda)dt
\geq \left(\lambda(1),\rho^1\right)-\left(\lambda(0), \rho^0\right).
\]\end{proof}

\begin{corollary}\label{cor:bar lambda} We have 
\begin{equation}\label{eq:minmax}
\sup_{\lambda\in B}\inf_{(\rho,m)\in A^{\infty}}\mathcal{L}(\rho,m,\lambda)=\inf_{(\rho,m)\in A^{\infty}}\sup_{\lambda\in B}\mathcal{L}(\rho,m,\lambda).
\end{equation}
\end{corollary}
\begin{proof}
Since $A^1\subset A^{\infty}$,  
\begin{equation*}
\sup_{\lambda\in B}\inf_{(\rho,m)\in A^{\infty}}\mathcal{L}(\rho,m,\lambda)\leq\sup_{\lambda\in B}\inf_{(\rho,m)\in A^1}\mathcal{L}(\rho,m,\lambda).
\end{equation*}
Define 
\begin{equation*}
\mathcal{L}^*(\lambda)=\inf_{(\rho,m)\in A^1}\mathcal{L}(\rho,m,\lambda).
\end{equation*}
Assume that $\{\lambda_n\}_n\subset B$ is a maximizing sequence such that
\begin{equation}\label{eq:nov03.2017.5.5}
\sup_{\lambda\in B}\mathcal{L}^*(\lambda)=\lim_{n\to+\infty}\mathcal{L}^*(\lambda_n).
\end{equation}
By Proposition \ref{prop:bar lambda}, there exists $\{\bar \lambda_n\}_n\subset B$ such that 
\begin{equation}\label{eq:nov03.2017.6}
\mathcal{L}^*(\bar\lambda_n)=\mathcal{L}^*(\lambda_n) \quad \text{and} \quad \left(H\Bigl(\dot {\bar\lambda}_n,\nabla_G\bar\lambda_n\Bigr)\right)_+=0.
\end{equation}
Combining \eqref{eq:nov03.2017.5.5} and \eqref{eq:nov03.2017.6} and using Lemma \ref{lem:ALAinfty} we obtain 
\begin{eqnarray*} 
\sup_{\lambda\in B}\inf_{(\rho,m)\in A^1}\mathcal{L}(\rho,m,\lambda) 
&=&\lim_{n\to\infty}\mathcal{L}^*(\lambda_n)\\
&=&\lim_{n\to\infty}\left(\bar\lambda_n(1),\rho^1\right)-\left(\bar\lambda_n(0),\rho^0\right)\\
&\leq&\sup_{\lambda\in B}\left\{\left(\lambda(1),\rho^1\right)-\left(\lambda(0),\rho^0\right)\;\; \big| \;\; H(\dot \lambda,\nabla_G\lambda)\leq 0\right\}\\
&=&\sup_{\lambda\in B}\inf_{(\rho,m)\in A^{\infty}}\mathcal{L}(\rho,m,\lambda).
\end{eqnarray*}
Since $A^1 \subset A^{\infty}$, using the duality result in Lemma \ref{le:minmax one bdd}, we have proven that 
\[
 \sup_{\lambda\in B}\inf_{(\rho,m)\in A^{\infty}}\mathcal{L}(\rho,m,\lambda)=\sup_{\lambda\in B}\inf_{(\rho,m)\in A^1}\mathcal{L}(\rho,m,\lambda)=\inf_{(\rho,m)\in A^1}\sup_{\lambda\in B}\mathcal{L}(\rho,m,\lambda).
\]
We exploit once more the fact that $A^1 \subset A^{\infty}$ to infer 
\[
 \sup_{\lambda\in B}\inf_{(\rho,m)\in A^{\infty}}\mathcal{L}(\rho,m,\lambda) \geq \inf_{(\rho,m)\in A^\infty}\sup_{\lambda\in B}\mathcal{L}(\rho,m,\lambda).
 \]
Since the reverse inequality always holds, we conclude the proof of the Corollary. \end{proof}

We now state the main result of this section. 
\begin{theorem}\label{th:dual problem} We have 
\begin{equation}\label{eq:dual problem}
\min_{(\rho, m) \in \mathcal C(\rho^0, \rho^1)}\left\{\mathcal A(\rho,m)\right\}=
\sup_{\lambda\in B}\left\{\left(\lambda(1),\rho^1\right)-\left(\lambda(0),\rho^0\right)\; | \;H(\dot \lambda ,\nabla_G\lambda)=0\right\}.
\end{equation}
\end{theorem}
\begin{proof} Define  
\begin{equation*}
I_{\mathcal C(\rho^0, \rho^1)}=
\left\{
\begin{aligned}
0 & \quad \text{if}\;\;(\rho, m) \in \mathcal C(\rho^0, \rho^1)\\
\infty & \quad \text{if}\;\;(\rho, m) \in A^\infty \setminus \mathcal C(\rho^0, \rho^1).
\end{aligned}
\right.
\end{equation*}
By \eqref{eq:property-of-E-0} and \eqref{eq:property-of-E}, for any $(\rho, m) \in A^\infty$, we have 
\[
\sup_{\lambda \in B} \mathcal L(\rho, m, \lambda)=\mathcal A(\rho, m)+ I_{\mathcal C(\rho^0, \rho^1)}(\rho, m).
\] 
Thus,  
\begin{equation}\label{eq:dual-new1}
\inf_{(\rho, m) \in A^\infty} \sup_{\lambda \in B} \mathcal L(\rho, m, \lambda)= \inf_{(\rho, m) \in A^\infty}\Bigl\{ \mathcal A(\rho, m)+ I_{\mathcal C(\rho^0, \rho^1)}(\rho, m)\Bigr\}.
\end{equation}
Since $\mathcal C(\rho^0, \rho^1) \subset A^\infty$, exploiting \eqref{eq:dual-new1}, we infer 
\begin{equation}\label{eq:dual-new2}
\min_{(\rho, m) \in \mathcal C(\rho^0, \rho^1)}  \mathcal A(\rho, m)=  \inf_{(\rho, m) \in A^\infty}\Bigl\{ \mathcal A(\rho, m)+ I_{\mathcal C(\rho^0, \rho^1)}(\rho, m)\Bigr\}= \inf_{(\rho, m) \in A^\infty} \sup_{\lambda \in B} \mathcal L(\rho, m, \lambda).
\end{equation}
We first use Corollary \ref{cor:bar lambda} in \eqref{eq:dual-new2} to conclude that   
\[
\min_{(\rho, m) \in \mathcal C(\rho^0, \rho^1)}  \mathcal A(\rho, m)= \sup_{\lambda \in B} \inf_{(\rho, m) \in A^\infty} \mathcal L(\rho, m, \lambda).
\] 
We reach the desired conclusion by noting that in light of Lemma \ref{lem:ALAinfty} and Proposition \ref{lem:legimply=}, the right hand--side of this last identity is nothing but the supremum in \eqref{eq:dual problem}.   \end{proof}

%
%
\section{Ingredients for Duality in the non-smooth case} \label{section6}
Throughout this section, we further assume  $g$ satisfies \eqref{eq:finite energy}. We will make use of $H_0$, the restriction of the recession function of $H$ (cf. e.g. \cite{buttazzo89} ) to $\mathbb R^n \times \{0\}:$ 
\[
H_0(a)= \sup_{\rho \in \mathcal P(G)} (a; \rho)= \max_{1\leq i\leq n} a_i.
\] 
\begin{lemma}\label{le:recession1} Assume $\nu$ is a non--negative Borel regular measure on $(0,1)$ such that $\nu$ and $\mathcal L^1|_{(0,1)}$ are mutually singular. Let $m \in L^2(0,1; S^{n\times n})$, let  $\beta:(0,1) \rightarrow \mathbb R^n$ be a Borel map  (defined $\mathcal L^1$ a.e.) and let $g:(0,1) \rightarrow \mathbb R^n$ be a Borel map  (defined $\nu$ a.e.).  Then the following assertions are equivalent 
\begin{enumerate}
\item[(i)] 
\begin{equation}\label{eq:recession2}
\begin{cases}
H(\beta,m) & \leq 0 \quad \mathcal L^1 \quad \text{a.e. in $(0,1)$}\\
H_0(g) & \leq 0 \quad \nu \quad \text{a.e. in $(0,1)$.}
\end{cases}
\end{equation} 
\item[(ii)] For every non--negative function $\varphi \in C([0,1]),$ we have $s[\varphi] \leq 0$ if we set 
\[
s[\varphi]:=\sup_{\rho} \biggl\{\int_0^1 \Bigl( (\rho, \beta) dt+\bigl(\rho, g \nu(dt)\bigr) +{1\over 2} \|m\|_\rho^2 dt \Bigr)\varphi(t)\; \Big| \; \rho \in \mathcal C_G \biggr\}.
\]
Here, $\mathcal C_G$ is the set of Borel maps of  $(0,1)$ into $\mathcal P(G).$
\end{enumerate} 
\end{lemma}
\proof{} Let $\varphi \in C([0,1])$ be non--negative. Since $\nu$ and $\mathcal L^1|_{(0,1)}$ are mutually singular  
\[
s[\varphi]= \sup_{\rho} \biggl\{\int_0^1 \bigl( (\rho, \beta)+{1\over 2} \|m\|_\rho^2 \bigr)\varphi(t)dt\; \Big| \; \rho \in \mathcal C_G\biggr\} +
\sup_{\rho} \biggl\{\int_0^1 \bigl(\rho, g \nu(dt)\bigr)  \varphi(t)\; \big| \; \rho \in \mathcal C_G \biggr\}.
\] 
Thus, 
\[
s[\varphi]=\int_0^1 H( \beta, m)\varphi(t)dt+  \int_0^1 H_0( g)\varphi(t)\nu(dt).
\]
Using again the fact that $\nu$ and $\mathcal L^1|_{(0,1)}$ are mutually singular, we conclude that $s[\varphi] \leq 0$ for all non--negative $\varphi \in C([0,1])$ if and only if \eqref{eq:recession2} holds. \endproof

\begin{remark}\label{re:recession3}  Let $R:[0,1] \rightarrow [0,1]$ be a Lipschitz function and let $h \in L^2(0,1)$ be a monotone non--decreasing function.  Observe that $h \in {\rm BV_{loc}}(0,1)$ and if $h(0^+)>-\infty$ then $h \in {\rm BV}(0, 3/4)$ and if $h(1^-)<+\infty$ then $h \in {\rm BV}(1/4, 1)$. Recall that when $h(0^+)$ is finite, it is the trace of $h$ at $0$ and will simply be denoted as $h(0).$ Similarly, we denote as $h(1)$ the trace of $h$ at $1$ when it exists. When $R(0)=0$, even if $h(0^+)=-\infty$, we interpret $R(0) h(0)$ as $0$. Similarly if $R(1)=0$, even if $h(1^-)=\infty$, in which case, we interpret $R(1) h(1)$ as $0$. We have   
\begin{equation}\label{eq:recession1.5} 
h(1) R(1)-h(0)R(0)= \int_0^1 \dot R h dt +  \int_0^1  R \dot h(dt).
\end{equation}
\end{remark}
\proof{} For each natural number $k$ we define the function $\varphi_k \in W^{1,\infty}_0(0,1)$ as 
\[\varphi_k(t):= 
\begin{cases}
k t &  \quad \text{if} \quad 0 \leq t \leq k^{-1}\\
 1 &  \quad \text{if} \quad {1\over k} < t< 1-k^{-1}\\
k(1-t) & \quad \text{if} \quad 1-k^{-1}<t \leq 1.
\end{cases}
\]
Since $R \varphi_k  \in W^{1,\infty}_0(0,1)$ we have 
\begin{equation}\label{eq:recession4.5} 
0=\int_0^1\Bigl( {d (R\varphi_k)\over dt} h dt + R \varphi_k \dot h(dt)\Bigr)= \int_0^1 \varphi_k \biggl(\dot R  h dt + R\dot h(dt)\biggr) +k\int_0^{1\over k} R h dt -k\int_{1-{1\over k}}^1 R h dt.
\end{equation} 
We use the dominated convergence theorem and then the monotone convergence theorem (since $\dot h$ is a Borel regular measure) to obtain 
\begin{equation}\label{eq:recession4} 
\lim_{k\rightarrow \infty} \int_0^1 \dot R \varphi_k h dt= \int_0^1 \dot R h dt \quad \text{and} \quad \lim_{k\rightarrow \infty} \int_0^1 R \varphi_k \dot h(dt)= \int_0^1 R \dot h(dt).
\end{equation}
Observe that 
\[
\Bigl| \int_0^{1\over k} R h dt- R(0) \int_0^{1\over k} h dt\Bigr| \leq { {\rm Lip}(R)\over k} \int_0^{1\over k} |h| dt, \quad 
\Bigl| \int_{1-{1\over k}}^1 R h dt- R(1) \int_{1-{1\over k}}^1 h dt\Bigr| \leq { {\rm Lip}(R)\over k} \int_{1-{1\over k}}^1 |h| dt.
\]
Since, when we use the above interpretation of $R(0)h(0)$ and $R(1)h(1)$ we have 
\[
 \lim_{k \rightarrow \infty}R(0) k \int_0^{1\over k} h dt= R(0) h(0) \quad \text{and} \quad   \lim_{k \rightarrow \infty} R(1) k\int_{1-{1\over k}}^1 h dt=R(1)h(1)
\]
we conclude that 
\begin{equation}\label{eq:recession5} 
\lim_{k\rightarrow \infty} k  \int_0^{1\over k} R h dt= R(0) h(0) \quad \text{and} \quad   \lim_{k \rightarrow \infty} k\int_{1-{1\over k}}^1 R h dt= R(1)h(1).
\end{equation} 
Combining (\ref{eq:recession4.5}--\ref{eq:recession5}) we verify \eqref{eq:recession1.5}.\endproof

\begin{definition}\label{de:bv-loc}
Let $ \lambda \in {\rm BV_ {loc}}(0,1; \mathbb R^n)$ such that the distributional derivative $\dot \lambda$ is the sum of an absolutely continuous part $\dot \lambda^{{\rm abs}}\mathcal L^1$ and a singular part (a Borel regular measure) $\dot \lambda^{{\rm sing}}$. Here, $\dot \lambda^{{\rm abs}}:(0,1) \rightarrow (-\infty, 0]^n$ is a Borel function. Choose a non--negative Borel regular measure $\nu$ such that $-\dot \lambda^{{\rm sing}}_i <<\nu$, and $\nu$ and $\mathcal L^1$ are mutually singular. We say that $\lambda$ belongs to ${\rm B}_*$ if 
\begin{equation}\label{eq:HJEabs}
H\left(\dot \lambda^{{\rm abs}} ,\nabla_G\lambda\right)=0\quad \mathcal L^1\quad\text{a.e. in $(0,1)$},
\end{equation}
and 
\begin{equation}\label{eq:HJEsingular}
\max_{i=1\leq i \leq n}\Bigl\{{d\dot \lambda^{{\rm sing}}_i \over d\nu}\Bigr\}= 0,\quad\nu\quad \text{a.e. in $(0,1)$}.
\end{equation}
\end{definition}

\begin{lemma}\label{le:recession7}   Let $\lambda \in L^2(0,1; \mathbb R^{n})$ be such that $\lambda_i$ is monotone non--increasing for any $i \in \{1, \cdots, n\}.$ Let $(\rho, m) \in \mathcal C(\rho^0, \rho^1)$ be such that $\rho$ is Lipschitz, $\rho_i(0)=0$ whenever $\lambda_i\not\in {\rm BV}(0,\, 0.75)$, and $\rho_i(1)=0$ whenever $\lambda_i\not\in {\rm BV}(0.25,\, 1)$. Let $\nu$ be the one in Definition \ref{de:bv-loc}. If 
\begin{equation}\label{eq:recession8}
\begin{cases}
H(\dot \lambda^{abs},\nabla_G \lambda) & \leq 0 \quad \mathcal L^1 \quad \text{a.e. in $(0,1)$}\\
H_0\Bigl({d \dot {\lambda}^{sing} \over d\nu}\Bigr) & \leq 0 \quad \nu \quad \text{a.e. in $(0,1)$,}
\end{cases}
\end{equation} 
then 
\[
\mathcal A(\rho, m) \geq (\lambda(1), \rho^1 )- (\lambda(0), \rho^0 ).
\]
Here, for each $i \in \{1, \cdots, n\}$, we have interpreted $\lambda_i(0) \rho^0_i $ and as $\lambda_i(1) \rho^1_i $ as in Remark \ref{re:recession3}. 
\end{lemma}
\proof{} To avoid trivialities, we assume that $\mathcal A(\rho, m)<\infty$, in which case for $\mathcal L^1$ almost every $t \in (0,1),$ $m_{ij}(t)=0$ if $g(\rho_i, \rho_j)(t)=0$. By Remark \ref{re:Legendre-trans} (ii) we have 
\begin{equation}\label{eq:recession9}
F(\rho, m)+\|b\|^2_\rho > 2 (m, b)
\end{equation} 
unless $m_{ij}= g(\rho_i, \rho_j) b_{ij}$ for all $(i, j) \in E.$ 
We have
\begin{eqnarray}
\int_0^1 {1\over 2}F(\rho, m)dt &\geq &\int_0^1 \biggl({1\over 2} F(\rho, m)dt +H(\dot \lambda^{abs}, \nabla_G \lambda)dt + H_0\Bigl({d \dot\lambda^{sing} \over d\nu}\Bigr)d\nu \biggr)\nonumber\\ 
&\geq & \int_0^1 \biggl({1\over 2} F(\rho, m)dt + (\dot \lambda^{abs}, \rho)dt +{1\over 2} \|\nabla_G \lambda\|^2_\rho dt+  \Bigl({d \dot\lambda^{sing} \over d\nu}, \rho\Bigr) d\nu \biggr)\nonumber\\
&= & \int_0^1 \biggl({1\over 2} F(\rho, m)dt + (\dot \lambda(dt), \rho)+{1\over 2} \|\nabla_G \lambda\|^2_\rho dt  \biggr)\label{eq:halfdualproblem}.
\end{eqnarray} 
We use Remark \ref{re:recession3}, then the fact that $(\rho, m) \in \mathcal C(\rho^0, \rho^1),$  to  obtain after integrating by parts,
\[
 \int_0^1 (\dot \lambda(dt), \rho) = -  \int_0^1 (\lambda, \dot \rho) dt + (\lambda(1), \rho^1)-(\lambda(0), \rho^0)=  \int_0^1 (\lambda, \nabla_G\cdot m) dt +  \bigl(\lambda(1), \rho^1\bigr)-\bigl(\lambda(0), \rho^0\bigr).
\]
This, together with \eqref{eq:halfdualproblem} yields 
\[
\mathcal A(\rho, m) \geq  \int_0^1 \biggl({1\over 2} F(\rho, m) -(\nabla_G \lambda,  m) dt +{1\over 2} \|\nabla_G \lambda\|^2_\rho  \biggr)dt + 
 \bigl(\lambda(1), \rho^1\bigr)-\bigl(\lambda(0), \rho^0\bigr).
\]
Thanks to \eqref{eq:recession9}, we reach the desired conclusions. \endproof

\begin{remark}\label{re:necessary-sufficient} Let $(\rho, m)$ and $\lambda$ be as in Lemma \ref{le:recession7}. 
\begin{enumerate}
\item[(i)] A necessary and sufficient condition to have $\mathcal A(\rho, m)=  \bigl(\lambda(1), \rho^1\bigr)-\bigl(\lambda(0), \rho^0\bigr)$ is that 
\begin{equation}\label{main1} (a) \quad m_{ij}= g(\rho_i, \rho_j) \bigl(\nabla_G\lambda\bigr)_{ij} \quad \mathcal{L}^1 \quad \text {a.e in $(0,1)$ for all} \quad  (i, j) \in E.
\end{equation}
\begin{equation}\label{main2} (b) \quad 
0=H(\dot \lambda^{abs}, \nabla_G \lambda)=  (\dot \lambda^{abs}, \rho)+ {1\over 2} \|\nabla_G \lambda\|^2_\rho \quad \mathcal L^1 \quad \text{a.e. in $(0,1)$},
\end{equation}
\begin{equation}\label{main3} 
0=H_0\Bigl({d \dot\lambda^{sing} \over d\nu}\Bigr)=\Bigl({d \dot\lambda^{sing} \over d\nu}, \rho\Bigr) \quad \nu\quad \text{a.e. in $(0,1)$.} 
\end{equation} 
\item[(ii)] For any $0\leq s < t \leq 1$ we have 
\[
(\lambda(t), \rho(t))=\min_{p \in \mathcal P(G)} \Bigl\{ (\lambda(0), p)+{1 \over 2(t-s)} \mathcal W_g^2(\rho(t), p)\Bigr\}
\]
\end{enumerate} 
\end{remark}

%
%
\section{Characterization of geodesics and extended Hamilton Jacobi Equations}\label{section7}
Throughout this section, we assume  \eqref{eq:finite energy} hold and characterize the minimizers of \eqref{eq:dual problem}. 

\begin{lemma}\label{lem:gral2bdd} Let $\lambda\in H^1\left(0,1;\mathbb R^{n}\right)$ be such that  
\begin{equation}\label{eq:HJE}
H\left(\dot \lambda,\nabla_G\lambda\right)=0\,\,\text{a.e. in $(0,1)$}.
\end{equation}
\begin{enumerate}
\item[(i)] We have $\dot \lambda_i\leq 0$ a.e. in $(0,1)$ for any $i\in V$.
\item[(ii)] If we further assume that $\gamma_P(\rho^0)$, $\gamma_P(\rho^1)>0$ and $C$ a constant  such that
\begin{equation*}
\left(\lambda(1),\rho^1\right)-\left(\lambda(0),\rho^0\right)\geq C
\end{equation*}
then there exists $C_0$ depending on $C$, $\gamma_P(\rho^0)$, $\gamma_P(\rho^1)$, $w$, $n$ such that $\|\nabla_G\lambda\|_{L^2(0,1)}\leq C_0$.
\end{enumerate} 
\end{lemma}
\begin{proof} (i) Since \eqref{eq:HJE} holds for $\lambda$ a.e. on $(0,1)$
\begin{equation}\label{eq:nov03.2017.7}
( \dot \lambda ,\rho)+\frac{1}{2}\|\nabla_G \lambda\|^2_\rho \leq 0\,\,\text{for any $\rho\in \mathcal{P}(G)$}.
\end{equation}
For $i \in \{1, \cdots, n\}$ fixed, we set $\rho_j=\delta_{ij}$ to discover that  (i) holds.

(ii) Let 
\begin{equation}\label{eq:nov03.2017.7new} 
\tilde\rho(t):=(1-t)\rho^0+t\rho^1 \in \mathcal C(\rho^0, \rho^1).
\end{equation} 
Let $\bar\rho:=(1/n, \cdots,1/n)\in \mathcal{P}_0(G)$. By Remark \ref{re:extension}, there is a unique $l_{\bar\rho}(\rho^1-\rho^0)$ such that
\[
\dot {\tilde\rho}=-\textrm{div}_{\bar\rho}\left(l_{\bar\rho}(\rho^1-\rho^0)\right)=-g\left(\frac{1}{n},\frac{1}{n}\right)\textrm{div}_G\left(l_{\bar\rho}(\rho^1-\rho^0)\right)\quad \text{and} \quad l_{\bar\rho}(\rho^1-\rho^0)\in T_{\bar\rho}\mathcal{P}(G).
\]
Note 
\begin{equation}\label{eq:nov03.2017.7new2} 
\tilde m := g({1\over n},{1\over n}) l_{\bar\rho}(\rho^1-\rho^0)\in L^2(0,1; S^{n\times n})
\end{equation}  
satisfies 
\begin{equation}\label{eq:nov03.2017.8}
\dot {\tilde\rho}+\textrm{div}_G(\tilde m)=0.
\end{equation}
Using Lemma \ref{lem:concave}, we have
\begin{equation}\label{eq:nov03.2017:8.5}
\gamma_P(\tilde\rho(t)) \geq (1-t)\gamma_P(\rho^0)+t\gamma_P(\rho^1)\geq\min\{\gamma_P(\rho^0), \gamma_P(\rho^1)\}=:\epsilon_1.
\end{equation}
Using  \eqref{eq:nov03.2017.8} we obtain 
\begin{equation}\label{eq:nov03.2017.9}
C\leq (\lambda(1),\rho^1)- (\lambda(0),\rho^0 )=\int_0^1  (\dot \lambda ,\tilde\rho)+ (\lambda,\dot {\tilde \rho}) dt=\int_0^1   (\dot \lambda ,\tilde\rho)- (\lambda, \textrm{div}_G(\tilde m)) dt.
\end{equation}
Setting $\tilde \lambda_i:= \lambda_i -1/n\sum_{j=1}^n \lambda_j$,  using Lemma \ref{le:Poincare-inequality},  we observe that \eqref{eq:nov03.2017.7}, together with \eqref{eq:nov03.2017.9}  implies 
\begin{eqnarray*}
C\leq \int_0^1\biggl(-\frac{1}{2}\|\nabla_G \tilde \lambda\|^2_{\tilde \rho} - (\tilde \lambda, \textrm{div}_G(\tilde m))\biggr) dt
&\leq &  \int_0^1\biggl(-{\gamma_P(\tilde\rho)\over 2}\|\tilde \lambda\|^2 - (\tilde \lambda, \textrm{div}_G(\tilde m)) \biggr) dt\\
&\leq &\int_0^1\biggl(-{\epsilon_1\over 2}\|\tilde \lambda\|^2 - (\tilde \lambda, \textrm{div}_G(\tilde m)) \biggr) dt.
\end{eqnarray*} 
This reads off 
\[
{\epsilon_1\over 2}    \int_0^1 \|\tilde \lambda\|^2 dt + \int_0^1 (\tilde \lambda, \textrm{div}_G(\tilde m)) dt \leq -C.
\]
Therefore,  there exists a constant $C_1$ depending only on $C$, $\epsilon_1$ and $\int_0^1\|\textrm{div}_G(\tilde m)\|^2 dt$ such that  
\[
 \int_0^1 \|\tilde \lambda\|^2 dt \leq C_1.
\]
Hence, there exists $C_0$ depending on $C$, $\gamma_P(\rho^0)$, $\gamma_P(\rho^1)$, $w$ and $n$ such that 
\[
 \int_0^1 \|\nabla_G \tilde \lambda\|^2 dt \leq C_0.
\]
Since $\nabla_G \tilde \lambda=\nabla_G \lambda$ we conclude the proof. 
\end{proof}

\begin{remark}\label{rem:L2bdd} Since $\rho^0, \rho^1 \in \mathcal P(G)$, there are  $i, j \in \{1, \cdots, n\}$ such that $\rho^0_i, \rho^1_j>0.$ The following lemma draws consequence from these facts.
\end{remark}

\begin{lemma}\label{lem:L2bdd}
Assume $\gamma_P(\rho^0), \gamma_P(\rho^1)>0$ and $\lambda$ is as in Lemma \ref{lem:gral2bdd} and satisfies (ii) of that lemma.  Then there exists some constant $C_1$ depending on $C$, $\gamma_P(\rho^0)$, $\gamma_P(\rho^1)$, $w$, $n$ such that
\begin{itemize}
\item [(i)] if $\rho_i^0\geq \epsilon>0$ for some $i\in V$, $\delta \|\dot \lambda_i\|_{L^1(0,1-\delta)} \leq C_1 \epsilon^{-1}$ for any $0<\delta<1$;
\item [(ii)] if $\rho_i^1\geq \epsilon>0$ for some $i\in V$, $\delta \|\dot \lambda_i\|_{L^1(\delta,1)}  \leq C_1 \epsilon^{-1}$ for any $0<\delta<1$;
\item [(iii)] if $\rho_i^0,\rho_i^1\geq \epsilon>0$ for some $i\in V$, $\|\dot \lambda_i\|_{L^1(0,1)} \leq C_1 \epsilon^{-1}$;
\item [(iv)] $\|\lambda\|_{L^2(0,1)}\leq C_1$.
\end{itemize}
\end{lemma}
\begin{proof}
Let $\left(\tilde\rho(t),\tilde m(t)\right)$ be the pair defined in \eqref{eq:nov03.2017.7new} and \eqref{eq:nov03.2017.7new2}. We exploit \eqref{eq:nov03.2017.9} to obtain 
\[
C \leq \int_0^1 (\dot \lambda ,\tilde\rho) dt+\|\nabla_G\lambda\|_{L^2(0,1)}\|\tilde m\|_{L^2(0,1)}.
\]
We use Lemma \ref{lem:gral2bdd}  to obtain a constant $C_1$ depending on $C$, $\gamma_P(\rho^0)$, $\gamma_P(\rho^1)$, $w$, $n$ such that
\[
- \int_0^1 (\dot \lambda ,\tilde\rho) dt \leq C_1.
\]
From one line to another, we may increase the value of $C_1$ when necessary. 

(i) By Lemma \ref{lem:gral2bdd}, $\dot \lambda_j \leq 0$ for any $j \in \{1, \cdots, n\}$ and so, $-\dot \lambda_i \tilde\rho_i \leq  -(\dot \lambda ,\tilde\rho).$  Thus, if $\rho_i^0\geq\epsilon>0$ then $\tilde \rho_i\geq \delta \epsilon$ on $t\in [0,1-\delta].$ Hence, 
\begin{equation*}
\int_0^{1-\delta}\delta\epsilon|\dot \lambda_i|dt\leq-\int_0^1\tilde \rho_i(t)\dot \lambda_i dt\leq- \int_0^1 (\dot \lambda ,\tilde\rho) dt \leq C_1. 
\end{equation*}

The proofs  of (ii) and (iii) follow the same lines of argument as that of (i).

(iv) Observe that as $\rho^0, \rho^1\in\mathcal{P}(G)$ there exist $i,j\in V$ such that $n\rho_i^0, n\rho_j^1\geq1$. Adding a constant to $\lambda_i$ if necessary, without loss of generality we assume that $\lambda_i(0)=0$. By (i), we have
\begin{equation}\label{eq:Linfty0delta}
|\lambda_i(t)|=\int_0^t |\dot \lambda_i  |ds \leq C_1
\end{equation}
for $t\in(0, 0.75]$. Since the graph $G$ is connected and $\|\nabla_G\lambda\|_{L^2}\leq C_0$, we obtain, for a bigger constant we still denote as $C_1$, 
\begin{equation}\label{eq:lambda034}
\|\lambda\|_{L^2(0, \; 0.75)}\leq C_1.
\end{equation}
Therefore, the set $\mathcal T$ of $t \in [0.25,\; 0.75]$ such that $|\lambda|\leq 2C_1$ is a set of positive measure. Using (ii), we have 
\begin{equation}\label{eq:lambda034.1}
\int_{t_0}^t |\dot \lambda_j |ds \leq C_1
\end{equation}
if $0.25\leq t_0\leq t \leq 1.$ In particular, taking $t_0 \in \mathcal T$ and increasing the value of $C_1$, \eqref{eq:lambda034.1} implies 
\begin{equation}\label{eq:lambda034.2}
\| \lambda_j \|_{L^\infty(0.25, \; 1)} \leq C_1.
\end{equation} 
As above, we use again the fact that the graph $G$ is connected and $\|\nabla_G\lambda\|_{L^2(0,1)}\leq C_0$ to obtain 
\begin{equation}\label{eq:lambda034.3}
\|\lambda\|_{L^2(0.25, \; 1)}\leq C_1.
\end{equation}
This, together with \eqref{eq:lambda034} proves (iv).
\end{proof}

\begin{theorem}\label{th:main-thm1} Let ${\rm B}_*$ be as in Definition \ref{de:bv-loc}. Assume $\gamma_P(\rho^0), \gamma_P(\rho^1)>0$.  
\begin{enumerate}
\item[(i)] We have 
\[
\min_{(\rho, m)} \Bigl\{\mathcal A(\rho, m)\; \Big| \; (\rho, m) \in \mathcal C(\rho^0, \rho^1) \Bigr\}=\sup_{\lambda}\left\{\left(\lambda(1),\rho^1\right)-\left(\lambda(0),\rho^0\right)\; \Big| \; \lambda\in {\rm B}_*\right\}.
\]
\item[(ii)] Then there exists  $ \lambda^* \in {\rm B}_*$ such that 
\[
\left(\lambda^*(1),\rho^1\right)-\left(\lambda^*(0),\rho^0\right)=\sup_{\lambda\in B}
\left\{\left(\lambda(1),\rho^1\right)-\left(\lambda(0),\rho^0\right)\; | \; H (\dot \lambda ,\nabla_G\lambda)=0\right\}.
\]
\end{enumerate}
\end{theorem}
\begin{proof} (i) Since ${\rm B}_* \cap B$ is a subset of ${\rm B}_*$, Lemma \ref{le:recession7} and Theorem \ref{th:dual problem} imply 
\begin{equation}\label{eq:dec10.2017.1}
\min_{(\rho, m)} \Bigl\{\mathcal A(\rho, m)\; \Big| \; (\rho, m) \in \mathcal C(\rho^0, \rho^1) \Bigr\}=\sup_{\lambda}\left\{\left(\lambda(1),\rho^1\right)-\left(\lambda(0),\rho^0\right)\; \Big| \; \lambda\in {\rm B}_*\right\}.
\end{equation}

(ii) Let $\{\lambda^l\}_l\subset B$ be the maximizing sequence of
\begin{equation*}
\sup_{\lambda\in B}\left\{\left(\lambda(1),\rho^1\right)-\left(\lambda(0),\rho^0\right)\;\; \Big| \;\;H(\dot \lambda ,\nabla_G\lambda)=0\right\}=:M.
\end{equation*}
Without loss of generality we can assume that 
\begin{equation}\label{eq:lower bdd1}
M-1\leq (\lambda^l(1),\rho^1)-(\lambda^l(0),\rho^0).
\end{equation}
Lemma \ref{lem:L2bdd}  ensures $\|\lambda^l\|_{L^2(0,1)}\leq C$ for a constant $C$ independent of $l.$ Therefore, for any $i\in V$ and any integer $k \geq 2$, there exist $s_{i,k}\in [0,\frac{1}{k}]$ and $\tilde s_{i,k}\in [1-\frac{1}{k},1]$ such that
\begin{equation*}
\max\left\{\left|\lambda_i^l(s_{i,k})\right|,\left|\lambda_i^l(\tilde s_{i,k})\right|\right\}\leq k C.
\end{equation*}
Since by Lemma \ref{lem:gral2bdd}, $\dot \lambda_i^l \leq 0$ a.e. in $(0,1)$ for any $i\in V$, we obtain 
$$ 
\Bigl\|\dot \lambda_i^l\Bigr\|_{\rm L^1\bigl(\frac{1}{k}, 1-\frac{1}{k}\bigr)}+ 
\Bigl\|\lambda_i^l\Bigr\|_{\rm L^1\bigl(\frac{1}{k}, 1-\frac{1}{k}\bigr)}=:\Bigl\|\lambda_i^l\Bigr\|_{{\rm BV} \bigl(\frac{1}{k}, 1-\frac{1}{k}\bigr)}\leq 2Ck+(k-2)C. 
$$
Thus, there is an increasing sequence $(n_l) \subset \mathbb N$ and $\lambda \in {\rm BV_{loc}}(0,1)$ such that 
\begin{itemize} 
\item [(i)] $(\lambda^{n_l})_l$ converges weakly to $\lambda^*$ in $L^2(0,1)$, in ${\rm BV_{loc}}(0,1)$ and strongly in $L^2_{loc}(0,1)$. 
\item [(ii)]  for any $k\in\mathbb N$, $\|\lambda_i^*\|_{{\rm BV}\bigl(\frac{1}{k},1-\frac{1}{k}\bigr)}\leq 3Ck$ for any $i\in V$. 
\end{itemize}
We denote the singular part of $-\dot \lambda_i^*$ as $-\dot \lambda^{*{\rm sing}}_i$ and denote the absolutely continuous part as $-\dot \lambda^{*{\rm abs}}_i \mathcal L^1$. Let $I$ be the set of $i \in \{1, \cdots, n\}$ such that $\rho^0_i>0$ and let $J$ be the set of $i \in \{1, \cdots, n\}$ such that $\rho^1_i>0.$ By Lemma \ref{lem:L2bdd} $(\lambda_i^{n_l})_l$ is a bounded sequence in ${\rm BV}(0,\, 0.75)$ for any $i \in I$ and is a bounded sequence in ${\rm BV}(0.25,\, 1)$ for any $i \in J.$ By the convergence of traces of functions of bounded variations, we may assume that 
\begin{equation}\label{eq:convergenceBV1}
\lim_{l\to+\infty} \lambda^{n_l}_i(0)=  \lambda^{*}_i(0) \quad \text{and} \quad \lim_{l\to+\infty} \lambda^{n_l}_j(1)=  \lambda^{*}_j(1) \quad \forall (i, j) \in I \times J.
\end{equation}

Let $\mathcal C_G$ be the set of Borel maps of  $(0,1)$ into $\mathcal P(G)$. Fix  $\varphi\in C_c^0(0,1)$ nonnegative, $\rho \in \mathcal C_G$ and  $m \in L^2(0,1; S^{n \times n})$ such that $\mathcal A(\rho, m)<\infty$. We use Remark \ref{re:Legendre-trans} (ii) to infer 
\begin{eqnarray*}
0&=&\varliminf_{l\to+\infty}\int_0^1\varphi(t)H\left(\frac{d\lambda^{n_l}}{dt},\nabla_G\lambda^{n_l}\right)dt\\
&\geq&\varliminf_{l\to+\infty}\int_0^1\varphi(t)\left( (\dot \lambda^{n_l} ,\rho)-\frac{1}{2} F(\rho, m) +(m, \nabla_G \lambda^{n_l})\right)dt \\
&=&\int_0^1\varphi(t)\left( (\dot \lambda^{*} ,\rho)-\frac{1}{2} F(\rho, m) +(m, \nabla_G \lambda^{*})\right)dt .
\end{eqnarray*}
Setting $m_{ij}=g(\rho_i, \rho_j)  \bigl( \nabla_G \lambda^{*}\bigr)_{ij}$ we use Remark \ref{re:Legendre-trans} (ii) again to conclude that 
\[
0\geq \int_0^1\varphi(t)\left( (\dot \lambda^{*} ,\rho)+\frac{1}{2} \|\nabla_G \lambda^{*}\|^2_\rho\right)dt.
\]
Since $\varphi$ and $\rho$ are arbitrary, we use Lemma \ref{le:recession1} to verify that 
\[
\begin{cases}
H\bigl(\dot \lambda^{*{\rm abs}}, \nabla_G \lambda^{*} \bigr) & \leq 0 \quad \mathcal L^1 \quad \text{a.e. in (0,1)}\\
H_0(\frac{d\dot \lambda^{*{\rm sing}}}{d\nu}) & \leq 0 \quad \nu \quad \text{a.e. in (0,1)}
\end{cases}
\]
where $\nu$ is a non--negative Borel regular measure such that $-\dot \lambda^{{\rm sing}}_i <<\nu$, and $\nu$ and $\mathcal L^1$ are mutually singular.
Thanks to \eqref{eq:convergenceBV1} and since $\{\lambda^n\}_n\subset B$ is a maximizing sequence, we have  
\[
M=\lim_{l\to+\infty} (\lambda^{n_l}(1),\rho^1 )-(\lambda^{n_l}(0),\rho^0 )=  
 \sum_{j\in J} \lambda^{*}_j(1) \rho^1_j-\sum_{i\in I} \lambda^{*}_i(0) \rho^0_i.
\]
In light of the convention in Remark \ref{re:recession3} we infer 
\begin{equation}\label{eq:toward-dual1}
M= (\lambda^{*}(1),\rho^1)-(\lambda^{*}(0) ,\rho^0).
\end{equation} 
Let $(\rho^*,m^*)\in H^1\left(0,1;\mathbb R^{n}\right)\times L^2\left(0,1; S^{n\times n}\right)$ be a minimizer of \eqref{eq:min m}. We combine \eqref{eq:dec10.2017.1} and  \eqref{eq:toward-dual1} to obtain
\[
\mathcal A(\rho^*, m^*)=  (\lambda^{*}(1),\rho^1)-(\lambda^{*}(0) ,\rho^0).
\]
This, together with Remark \ref{re:necessary-sufficient} (i) yields \eqref{eq:HJEabs} and \eqref{eq:HJEsingular}.\end{proof}

\begin{theorem}\label{th:duality-final} Assume $\gamma_P(\rho^0), \gamma_P(\rho^1)>0,$ $(\rho, m) \in \mathcal C(\rho^0, \rho^1)$ and $\mathcal A(\rho, m)<\infty$. A necessary and sufficient condition of $(\rho, m)$ to minimize $\mathcal A$ over   $ \mathcal C(\rho^0, \rho^1)$ is that  there exists $ \lambda \in {\rm BV_ {loc}}(0,1; \mathbb R^n) $ such that $\nabla_G \lambda$ and the distributional derivative $\dot \lambda,$ which is the sum of an absolutely continuous part $\dot \lambda^{\rm abs}\mathcal L^1$ and a singular part $\dot \lambda^{\rm sing}$, satisfy \eqref{main1}-\eqref{main3} 
\end{theorem} 
\proof{} Suppose $(\rho, m)$ to minimize $\mathcal A$ over   $ \mathcal C(\rho^0, \rho^1)$.  By Theorem \ref{th:main-thm1}, there is $\lambda \in {\rm BV_ {loc}}(0,1; \mathbb R^n)$ such that
\[
\left(\lambda(1),\rho^1\right)-\left(\lambda(0),\rho^0\right)=\mathcal A(\rho, m).
\] 
We use  Remark \ref{re:necessary-sufficient} (i) to conclude that \eqref{main1}-\eqref{main3} hold.

Conversely, suppose there is $\lambda \in {\rm BV_ {loc}}(0,1; \mathbb R^n)$ such that \eqref{main1}-\eqref{main3} hold. Relying on Remark \ref{re:necessary-sufficient} (i), we conclude that $(\rho, m)$ to minimize $\mathcal A$ over   $ \mathcal C(\rho^0, \rho^1)$. \endproof
\hfill\break
\textbf{Acknowledgements.} The research of WG was supported by NSF grant DMS--17 00 202.

\end{document}